\documentclass{birkjour}
\usepackage{hyperref}
 \newtheorem{thm}{Theorem}[section]
 
 \newtheorem{lem}[thm]{Lemma}
 \newtheorem{prop}[thm]{Proposition}
 \theoremstyle{definition}
 \newtheorem{defn}[thm]{Definition}
 \theoremstyle{remark}
 \newtheorem{rem}[thm]{Remark}
 
 \numberwithin{equation}{section}

\newcommand{\R}{\mathbb{R}}

\begin{document}

%
%
%
%
%
%
%
%
%

\title{Normalized solutions for a class of gradient-type Schr\"odinger systems under Neumann boundary conditions in bounded domains\footnote{Corresponding author: Xiaojun Chang}}

\author[X. J. Chang et al.]{Xiaojun Chang}

\address{%
School of Mathematics and Statistics \& Center for Mathematics and Interdisciplinary Sciences\\
Northeast Normal University\\
Changchun Jilin 130024\\
P.R. China}

\email{changxj100@nenu.edu.cn}

\author{Yuxin Li}
\address{School of Mathematics and Statistics\\
Northeast Normal University\\
Changchun Jilin 130024\\
P.R. China}
\email{liyx097@nenu.edu.cn}

\author{Yohei Sato}
\address{Department of Mathematics\\
Saitama University\\
Shimo-Okubo 255, Sakura-ku Saitama-shi, 338-8570\\
Japan}
\email{ysato@rimath.saitama-u.ac.jp}

\author{Yuxuan Zhang}
\address{School of Mathematics and Statistics\\
Northeast Petroleum University\\
Daqing Heilongjiang 163318\\
P.R. China}
\email{zhangyx595@nenu.edu.cn}



\begin{abstract}
We investigate the existence of normalized solutions  to the gradient-type Schr\"odinger system
\begin{equation*}
\begin{cases}
-\Delta u+ V_1(x)u+\lambda u= uv^2  & {\rm in} \,~ \Omega,\\
-\Delta v+ V_2(x)v+\lambda v=  u^2v  & {\rm in} \,~ \Omega 
\end{cases}
\end{equation*}
subject to the mass constraint $\int_{\Omega}\left(|u|^2+|v|^2 \right)dx=a>0$ and Neumann boundary conditions, where
$\Omega\subset \mathbb{R}^3$ is a smooth bounded domain, each $V_i$ is continuous, and $\lambda$ is a Lagrange multiplier.
Applying a minimax principle that incorporates Morse index information, we establish the existence of nontrivial normalized solutions of mountain pass type. The proof is based on a refined blow-up analysis adapted to such gradient-type systems, together with new Liouville-type theorems for finite Morse index solutions of the associated limit  systems in $\mathbb{R}^3$ and $\mathbb{R}^3_+$.
\end{abstract}

\maketitle

\markboth{X. J. Chang et al.}{Normalized solutions for gradient-type systems}

\section{Introduction}
This paper is devoted to the study of the existence of
normalized solutions to the gradient-type Schr\"odinger system under Neumann boundary conditions:
\begin{equation}\label{grad}
\begin{cases}
-\Delta u+V_1(x)u+ \lambda u=  uv^2,  & x\in \, \Omega,\\
-\Delta v+V_2(x)v+ \lambda v=  u^2v,  & x\in \, \Omega,\\
\frac{\partial u}{\partial \nu}=\frac{\partial v}{\partial \nu}=0,  \, &x\in\partial \Omega
\end{cases}
\end{equation}
under the mass constraint
\begin{equation}\label{grad2}
	\int_{\Omega}\left(|u|^2+|v|^2\right)dx=a,
\end{equation}
where $\Omega\subset \mathbb{R}^3$ is a smooth bounded domain, $V_i: \overline{\Omega} \to \mathbb{R}, i=1,2,$ are continuous functions, $\nu$ stands for the unit outer normal vector to $\partial \Omega$, $a$ is a positive constant, and the parameter $\lambda$ serves as a Lagrange multiplier.

The study of normalized solutions for nonlinear Schr\"odinger (NLS) equations and systems has attracted considerable attention, largely due to their important physical applications in areas such as Bose-Einstein condensates and nonlinear optics \cite{F2010,Malomed2008}. Research in this area has flourished since Jeanjean's pioneering work on the $L^2$-supercritical NLS equation \cite{Jeanjean1997}.
We refer the interested reader to \cite{AJM2022,Bartsch2017,BM2020,BMRV2021,CT2024,IM2020,IkNo,JJLV2022,JL2022,JL2020,LM2024,MRV2022,WW-2022,ZZ2022} for results on $\mathbb{R}^N$, and to \cite{CDGS2025, NTV-2014,NTV-2015,NTV-2019,PV-2017,PVY-2025,SZ2025} for those on bounded domains. For related work on metric graphs, we refer to \cite{CJS2024,BCJS2023,DJS2024} and the references therein.

It is well known that any weak solution $(u,v)$ of \eqref{grad}-\eqref{grad2} corresponds to a critical point in the space $H^1(\Omega)\times H^1(\Omega)$ of the functional
\[J(u,v):= \int_{\Omega}\left(|\nabla u|^2+ |\nabla v|^2\right) dx+\int_{\Omega}V_1(x)u^2dx+\int_{\Omega}V_2(x)v^2dx-\int_{\Omega}  u^2v^2 dx\]
restricted to the set
\[\mathcal{S}_a:=\{(u,v)\in H^1(\Omega)\times H^1(\Omega) : \int_{\Omega} \left(|u|^2+|v|^2\right)dx= a\}. \]
From the Gagliardo-Nirenberg inequality (see Lemma \ref{gnineq}), we know that the system \eqref{grad} exhibits $L^2$-supercritical characteristics, which implies that the associated functional $J$ is unbounded from below on the constraint set $\mathcal{S}_a$.

When $\Omega=\mathbb{R}^3$, Bartsch, Jeanjean and Soave \cite{BJS2016} investigated the following problem
\begin{equation}\label{bdddd}
\begin{cases}
-\Delta u+ \lambda_1 u=\mu_1 u^3+ \beta uv^2,  & x\in \mathbb{R}^3,\\
-\Delta v+ \lambda_2 v=\mu_2 v^3+ \beta u^2v,  & x\in \mathbb{R}^3,\\
\end{cases}
\end{equation}
with the prescribed mass constraints:
\begin{equation}\label{bdddd2}
	\int_{\mathbb{R}^3} |u|^2 dx =a_1>0,~~\int_{\mathbb{R}^3} |v|^2 dx =a_2>0.
\end{equation}
It was shown in \cite{BJS2016} that for arbitrary masses $a_i$ and positive parameter $\beta$, there exist thresholds $\beta_1, \beta_2>0$ such that the system \eqref{bdddd}-\eqref{bdddd2} admits a positive radial solution when $0<\beta<\beta_1$ or $\beta>\beta_2$. The case for $\beta<0$ was later addressed in 
\cite{Bartsch2017,BS2019}. Furthermore, Bartsch, Zhong and Zou \cite{BZW-2021} demonstrated that normalized solutions exist for any $a_1, a_2>0$ when $\beta>0$ lies within a broader range. For further recent results regarding \eqref{bdddd}-\eqref{bdddd2} and the explicit new ranges of $\beta>0$, we refer the reader to \cite{JZZ2024}. For studies on more general Schr\"odinger systems in $\mathbb{R}^N$, see \cite{BJ2018,GJ2018}.

 Guo et al. \cite{GuoLW-2019} considered the following system in $\mathbb{R}^2$:
\begin{equation}\label{diantai}
	\begin{cases}
		-\Delta u+  V_1(x)u+\lambda u= \mu_1 u^3+ \beta uv^2,  & x\in\mathbb{R}^2,\\
-\Delta v+ V_2(x)v+ \lambda v= \mu_2v^3+ \beta u^2v,  & x\in \mathbb{R}^2,\\
	\end{cases}
\end{equation}
subject to a different mass constraint:
\begin{equation}\label{vhgsqk}
	\int_{\mathbb{R}^2} \left(|u|^2+|v|^2\right) dx=1.
\end{equation}
The authors established criteria for the existence and non-existence of ground states for the system \eqref{diantai}-(\ref{vhgsqk}). Moreover, as $\beta$ approaches a critical value $\beta^*:= \mu^*+ \sqrt{(\mu^*-\mu_1)(\mu^*-\mu_2)}$,  they analyzed the uniqueness and symmetry-breaking of ground states under trapping potentials. Here
 $0<\mu_i<\mu^*:=\|w\|_2^2$, and $w$ denotes the unique positive solution of $-\Delta w+ w= w^3$ in $\mathbb{R}^2$. For further studies using this type of mass constraint, see \cite{GH2025,GY2023,GuoLW2019TAMS}.

 The study of normalized solutions to NLS equations in bounded domains was initiated by 
 Noris et al. \cite{NTV-2014} on the unit ball with Dirichlet boundary conditions, and later extended to  general bounded domains by Pierotti and Verzini \cite{PV-2017} in the $L^2$-supercritical setting. Subsequently, Noris, Tavares and Verzini \cite{NTV-2019} proved the existence of local minimum normalized solutions for the Sobolev critical case, a result recently generalized by \cite{PVY-2025,CLY2025} who established the existence of mountain pass type solutions. In the presence of a potential, Bartsch et al. \cite{BQZ2024} proved the existence of mountain pass type normalized solutions for NLS equations on sufficiently large star-shaped domains.

For Schr\"odinger systems with potentials on bounded domains, early results were obtained by Noris, Tavares and Verzini \cite{NTV-2015}, who studied the
existence of normalized solutions for system \eqref{diantai} under constraint \eqref{bdddd2} in $\Omega\subset \mathbb{R}^N$ with $N\le 3$. Later, in \cite{NTV-2019}, they 
extended their analysis to arbitrary dimensions ($N\ge 1$) and demonstrated the existence of local minimum normalized solutions.

All the aforementioned results on bounded domains concern Dirichlet boundary conditions. By contrast, the literature on normalized solutions with Neumann boundary conditions remains relatively scarce. 
 The first contribution in this direction is due to Pellacci et al. \cite{PPVV-2021}, who studied concentration phenomena for both types of boundary conditions.  Dovetta, Serra and Tentarelli \cite{DST2025} established non-uniqueness of normalized ground states of NLS equations with pure powers nonlinearities near the $L^2$-critical exponent on planar polygons
under Neumann boundary conditions. More recently, \cite{CRZ2025} proved the existence of mountain pass type normalized solutions for the autonomous 
$L^2$-supercritical NLS under Neumann boundary conditions.

In this paper, we investigate the existence of normalized solutions for the gradient-type Schr\"odinger system \eqref{grad}. Although ground state solutions and their concentration phenomena have been studied for such systems in $\mathbb{R}^N$ \cite{CS-2019,CS-2020} and on bounded domains \cite{SW-2015}, the corresponding normalized solution problem has remained unexplored and constitutes the main purpose of the present work.

Since the system is of $L^2$-supercritical, the associated energy functional is unbounded below on $\mathcal{S}_a$, which implies that the functional exhibits a mountain pass geometry.
This naturally leads us to seek normalized solutions of mountain pass type. The primary obstacle is the construction of a bounded Palais–Smale sequence from the mountain pass structure.  In $\mathbb{R}^N$, this is typically achieved by combining the Pohozaev identity with a global scaling argument. For problems on bounded domains, however, this strategy fails due to the loss of translational invariance and the presence of uncontrollable boundary terms in the Pohozaev identity. 
 Consequently, the techniques developed for $\mathbb{R}^N$ do not directly apply.

A further complication stems from the specific  structure of \eqref{grad}. In $\R^N$, system \eqref{grad} can be obtained as the formal limit of system \eqref{diantai} under the scaling $(u,v)\mapsto (\frac{u}{\sqrt{\beta}},\frac{v}{\sqrt{\beta}})$ as $\beta\to +\infty$. For systems like \eqref{diantai}, one can often reduce the problem to a single equation by exploiting semi-trivial solutions (where at least one component  identically 
zero), see \cite{BJS2016,JZZ2024}.  For the gradient-type Schr\"odinger system  \eqref{grad}, such reduction is impossible. Moreover,  semi-trivial solutions of \eqref{grad} exist only when the Lagrange multiplier $\tilde\lambda$ coincides with an eigenvalue of $-\Delta+V_i(x)$ under Neumann boundary conditions. The central challenge is therefore to construct nontrivial solutions (with both components nonzero), a task that requires the development of new analytical tools. For further background and recent advances on gradient-type Schr\"odinger systems, we refer to \cite{Alves-2007, Alves-2005}.

\medskip

Our main result is as follows.
\begin{thm}\label{Thm1.1}
There exists $\hat{a}^*>0$ such that for every $a\in (0, \hat{a}^*)$,  system \eqref{grad}-\eqref{grad2} admits a nontrivial normalized solution $(u^*,v^*)\in H^1(\Omega)\times H^1(\Omega)$ of mountain pass type for some $\lambda\in \mathbb{R}$.
\end{thm}

The proof of Theorem \ref{Thm1.1} adapts the approach developed in \cite{BCJS2024,CRZ2025}, which combines the monotonicity trick with Morse index information and blow-up analysis. This approach has since been applied in various contexts; see, for example, \cite{BCJS2023,CJS2024,WC2024}.

Our results extend the arguments in \cite{CRZ2025} from a single equation to gradient-type Schr\"odinger systems. A key ingredient is a blow-up analysis for the approximating solutions sequence $\{(u_n,v_n)\}$, which relies on their uniform Morse index information. We develop a systematic blow-up analysis framework for \eqref{grad}, which we believe can be naturally extended to a broader class of Schr\"odinger systems. 

Within this framework, we provide a robust argument to show that the solution sequence not only exhibits the standard exponential decay away from the boundary (see \cite{NT-1991, NT-1993}), but also uniform exponential decay away from the blow-up points. This uniform decay property is crucial for our analysis.
In addition, as a key part of this blow-up analysis, we establish new Liouville type theorems for finite Morse index solutions of the limiting systems in $\mathbb{R}^3$ and $\mathbb{R}^3_+$. The first such result for bounded finite Morse index solutions of nonlinear elliptic equations was 
proven by Bahri and Lions \cite{BL-1992}. Related Liouville theorems have since been obtained in many contexts, including scalar equations \cite{F-2007,HRS-1998-1, HRS-1998-2} and elliptic systems \cite{L-2021, YZ-2019}. To the best of our knowledge, however, none of the existing theorems directly apply to our setting. We therefore prove these Liouville-type theorems, which are adapted to the specific structure of \eqref{grad}.

Finally, to verify that the solution obtained is nontrivial, we investigate the asymptotic behavior of the Lagrange multiplier as the mass parameter
 $a$ tends to $0^+$.  This analysis relies on the observation that the corresponding energy remains uniformly bounded below by a positive constant for sufficiently small $a$.
Combining this lower bound with the Morse index information and properties of the eigenfunctions of  
 $-\Delta+V_i(x)$ under Neumann boundary condition, a contradiction argument yields the desired conclusion.

\medskip

The paper is organized as follows. In Section \ref{mp sol}, we first establish the existence of local minimizers and prove that the functionals $J_\rho$ (defined in \eqref{lajgrad}) admit a uniform mountain pass geometry for $\rho \in [\frac{1}{2},1]$. Then we obtain a mountain pass critical point $(u_\rho,v_\rho)$ on $\mathcal{S}_a$ for almost every such $\rho$. The nontriviality of solutions with small mass is also discussed. Section \ref{blow up ana} is devoted to a blow-up analysis of approximating solutions under the prescribed mass constraint and a uniform Morse index bound. Section \ref{proof of theorem} contains the proof of our main result, Theorem \ref{Thm1.1}. Finally, Appendix \ref{liouville thm} provides the proofs of the Liouville-type theorems for bounded solutions with finite Morse index to the gradient-type systems introduced in Section \ref{blow up ana}.

\section{Mountain pass type solutions for approximating problems}\label{mp sol}
In this section, we study a family of parameterized constraint functionals and establish the existence of mountain pass type normalized solutions for the approximating problems. Moreover, we provide some analysis to show that these solutions are nontrivial whenever the mass is sufficiently small.

To begin, we recall the following Gagliardo-Nirenberg inequality(see \cite{N-1966}).
\begin{lem}\label{gnineq}
For every $N\geq 3$ and $r\in(2,2^*)$, there exists a constant $C_{N,r,\Omega}$, depending on $N$, $r$ and $\Omega$, such that
\begin{equation*}
	\|u\|_r\leq C_{N,r,\Omega}\|u\|_2^{1-\gamma_r}\|u\|_{H^1}^{\gamma_r},~~\forall u\in H^1(\Omega),
\end{equation*}
where $\gamma_r:=\frac{N(r-2)}{2r}$.
\end{lem}

We introduce the parameterized functional $J_{\rho}: H^1(\Omega)\times H^1(\Omega)\to \mathbb{R}$ defined by
\begin{align}\label{lajgrad}
J_\rho(u,v)&:= \int_{\Omega}\left(|\nabla u|^2+ |\nabla v|^2 \right)dx+\int_{\Omega}V_1(x)u^2dx+\int_{\Omega}V_2(x)v^2dx\nonumber\\
&\quad- \rho\int_{\Omega} u^2v^2 dx,
\end{align}
where $\rho \in [\frac{1}{2},1]$.
For $m>0$, we also define $\mathcal{B}_m$ by 
\begin{equation*}
    \mathcal{B}_m:=\bigl\{ (u,v)\in \mathcal{S}_a : \int_{\Omega}\bigl( |\nabla u|^2+|\nabla v|^2\bigr)\,dx < m \bigr\}.
\end{equation*}
We have the following lemma.
\begin{lem}\label{biaoda}
For any $m>0$, there exists $a^*>0$ such that for any $a\in (0, a^*)$, we have
	\begin{eqnarray}\label{loc}
	\sup_{(u,v)\in \mathcal{B}_{\frac{m}{2}}} J_{\frac{1}{2}}(u,v) <m< 
	\inf_{(u,v)\in \partial\mathcal{B}_{2m}} J_\rho(u,v), ~~\forall \rho\in \left[\frac{1}{2}, 1\right].
\end{eqnarray}
\end{lem}
\begin{proof}
By  Lemma \ref{gnineq}, for any $(u,v) \in  \partial \mathcal{B}_{2m}$ we get
\begin{align}
	\int_{\Omega} u^2v^2 \,dx 
	& \leq \frac{1}{2}\left(\|u\|_4^4 + \|v\|_4^4\right) \notag \\
	& \leq \frac{1}{2}C^4_{3,4,\Omega}\left(\|u\|_2\|u\|_{H^1}^3 + \|v\|_2\|v\|_{H^1}^3\right) \notag \\
	& \leq C^4_{3,4,\Omega}a^{\frac{1}{2}}(4m+a)^\frac{3}{2}. \label{2-20-2}
\end{align}
By \eqref{2-20-2}, for any $(u,v)\in \partial \mathcal{B}_{2m}$ and $\rho\in [\frac{1}{2}, 1]$, 
we deduce
\begin{align*}
	J_\rho(u,v)
	& = \int_{\Omega}\left(|\nabla u|^2+ |\nabla v|^2\right)dx+\int_{\Omega}\left(V_1(x)u^2+V_2(x)v^2\right)dx\\
	 &\quad - \rho\int_{\Omega} u^2v^2dx \\
	& \ge \left(\|\nabla u\|_2^2 + \|\nabla v\|_2\right)- \left(\max_{i=1,2}\|V_i\|_{L^\infty}\right) \int_{\Omega}\left(u^2+v^2\right)dx \\
  &\quad- \int_{\Omega} u^2v^2dx \\
	& \ge 2m - \left(\max_{i=1,2}\|V_i\|_{L^\infty}\right)a -C^4_{3,4,\Omega}a^{\frac{1}{2}}(4 m+a)^{\frac{3}{2}}.
\end{align*}
On the other hand, for any $(u,v)\in \mathcal{B}_{\frac{m}{2}}$, we obtain
\begin{align*}
	J_{\frac{1}{2}}(u,v)
	& = \int_{\Omega}\left(|\nabla u|^2+ |\nabla v|^2\right)dx + \int_{\Omega}\left(V_1(x)u^2+V_2(x)v^2\right)dx\\
	 &\quad - \frac{1}{2}\int_{\Omega}u^2v^2dx \\
	& \le \|\nabla u\|_2^2 + \|\nabla v\|_2 + \left(\max_{i=1,2}\|V_i\|_{L^\infty}\right) \int_{\Omega}\left(u^2+v^2\right)dx \\
	& \le \frac{m}{2} + \left(\max_{i=1,2}\|V_i\|_{L^\infty}\right)a.
\end{align*}
Now we choose small $a^*>0$ such that 
\begin{align*}
	&\frac{m}{2} + \left(\max_{i=1,2}\|V_i\|_{L^\infty}\right)a^*< m \\
	&< 2m - \left(\max_{i=1,2}\|V_i\|_{L^\infty}\right)a^* -C^4_{3,4,\Omega}(a^*)^{\frac{1}{2}}(4 m+a^*)^{\frac{3}{2}}.
\end{align*}
Then, for any $a\in (0,a^*)$, inequalities (\ref{loc}) hold. This completes the proof.
\end{proof}

For any $m>0$ and $a\in (0,a^*)$, we consider the minimization problem
\begin{equation*}
	m'_\rho:=\inf_{(u,v)\in \mathcal{B}_{2m}}J_\rho(u,v).
\end{equation*}
By Lemma \ref{biaoda} and standard arguments, we can establish the existence of a local minimizer of $J_{\rho}$ within the set $\mathcal{B}_{2m}$.
\begin{lem}\label{existgrad}
Let $m>0$ and $a\in (0, a^*)$. 
Then $m'_\rho$ is achieved by some $(\tilde{u} _\rho, \tilde{v} _\rho)\in \mathcal{B}_{2m}$.
\end{lem}

We now prove that $J_{\rho}$ possesses a mountain pass geometry uniformly with respect to $\rho\in[\frac{1}{2},1]$.
\begin{lem}\label{uni mountain}
Let $m>0$ and $a\in (0, a^*)$.
Then there exist $(w_1, w_2) \in \mathcal{B}_{2m}$ and $ (w^*_1, w^*_2) \not\in \mathcal{B}_{2m}$ such that
\begin{equation*}
	c_\rho:= \inf_{\gamma\in \Gamma}\max_{t\in[0,1]}J_\rho(\gamma(t))
	>m>
	\max\{J_\rho(w_1, w_2), J_\rho(w^*_1, w^*_2) \} ,~~ \forall \rho\in\left[\frac{1}{2}, 1\right],
\end{equation*}
where
	\begin{equation}\label{Gammagrad}
		\Gamma:=\{\gamma\in C([0,1], \mathcal{S}_a) : \gamma(0)=(w_1, w_2), ~\gamma(1)=(w^*_1, w^*_2) \}.
	\end{equation}
\end{lem}
\begin{proof}
Denote by $B_r(x)$ the ball in $\mathbb{R}^3$ with radius $r$ and centered at $x$.
Let $\varphi \in C^\infty_0(B_1(0))$ be such that $\varphi>0$ within $B_1(0)$ and $\int_{B_1(0)}|\varphi|^2 dx=1$.
For $t>0$, $x_0\in \Omega$, define the functions
\begin{equation*}
	\varphi_t(x):=\left(\frac{a}{2}\right)^{\frac{1}{2}} t^\frac{3}{2} \varphi (t(x-x_0)),~~~
	x\in \Omega.
\end{equation*}
It can be verified that for sufficiently large $t$, the supports of $\varphi_t$ is contained within $ B_\frac{1}{t}(x_0)\subset \Omega$.
Moreover, we have $(\varphi_t, \varphi_t)\in \mathcal{S}_a$.
For any $\rho\in [\frac{1}{2},1]$, by taking $t$ sufficiently large, we obtain
\begin{eqnarray*}
\begin{aligned}
	J_\rho(\varphi_t,\varphi_t) 
	& = \int_{\Omega}\left(|\nabla \varphi_t(x)|^2+|\nabla \varphi_t(x)|^2\right)dx\\
	 &\quad + \int_{\Omega}\left(V_1(x) |\varphi_t(x)|^2+V_2(x)|\varphi_t(x)|^2\right)dx\\
	&\quad  - \rho\int_{\Omega} \varphi_t(x)^4 dx \\
	& \leq at^2\int_{B_1(0)}|\nabla \varphi|^2 + \frac{a}{2}\left(\|V_1\|_{L^\infty}+\|V_2\|_{L^\infty}\right) \int_{B_1(0)}|\varphi|^2dx\\
	 &\quad - \frac{a^2}{8}t^3\int_{B_1(0)}|\varphi|^4 dx<0.
\end{aligned}
\end{eqnarray*}
Hence, there exists $t_0>0$ large enough such that
\begin{equation*}
	J_\rho(\varphi_{t_0}, \varphi_{t_0})< m,~~\forall \rho\in \left[\frac{1}{2}, 1\right] 
	~~\text{ and }~~ (\varphi_{t_0}, \varphi_{t_0})\not\in \mathcal{B}_{2m}.
\end{equation*}
Let $(w_1,w_2)=(\tilde{u}_{\frac{1}{2}},\tilde{v}_{\frac{1}{2}})\in \mathcal{B}_{\frac{m}{2}}$ 
and $(w^*_1,w^*_2)=(\varphi_{t_0}, \bar{\varphi}_{t_0})$.
By the continuity of the functional $J_{\rho}$ along any path
 $\gamma \in \Gamma$, there exists $t_\gamma \in [0,1]$ such that $ \gamma(t_\gamma)\in \partial \mathcal{B}_{2m}$.
Consequently, by Lemma \ref{biaoda}, it follows that
\begin{align*}
	\max_{t\in[0,1]}J_\rho(\gamma(t))&\geq J_\rho(\gamma(t_\gamma))\geq \inf_{(u,v)\in \partial \mathcal{B}_{2m}}J_\rho(u,v)
	> m \\
	&> \max\{J_\rho(w_1,w_2), J_\rho(w_1^*,w_2^*)\},~~\forall \rho \in \left[\frac{1}{2}, 1\right].	
\end{align*}
The proof is complete.
\end{proof}
\begin{rem}
	We note that from the choice of $(\tilde{u}_{\frac{1}{2}},\tilde{v}_{\frac{1}{2}})$, $(\varphi_{t_0}, \bar{\varphi}_{t_0})$ and $\Gamma$, we can observe that these functions remain independent of the parameter $\rho$.
\end{rem}

\begin{defn}\label{Morseindexgrad}
For a domain $D$ and $(u,v)\in H_{loc}^1(D)\times H_{loc}^1(D)$ satisfying the following system
	\begin{eqnarray*}
\begin{cases}
-\Delta u+V_1(x)u+ \lambda u= \rho uv^2,  & x\in D,\\
-\Delta v+V_2(x)v+ \lambda v= \rho u^2v,  & x\in D,\\
\frac{\partial u}{\partial \nu}=\frac{\partial v}{\partial \nu}=0, &x\in \partial D,\\
\end{cases}
\end{eqnarray*}
where $\lambda \in \mathbb{R}$ and $\rho\in[\frac{1}{2},1]$, we define the quadratic form
\begin{equation*}
\begin{split}
		Q_{u,v}(\phi,\psi; D):= &\int_{D}\left(|\nabla \phi |^2+|\nabla \psi|^2 \right)dx+\int_{D}\left(V_1(x)|\phi |^2+V_2(x)|\psi|^2 \right)dx\\
    & + \lambda\int_{D}\left(|\phi |^2+|\psi|^2\right) dx \\
		&-\rho\int_{D} \left(v^2\phi^2+ u^2\psi^2+ 4 uv\phi\psi \right)dx.
\end{split}
\end{equation*}
The Morse index of $(u,v)$, denote by $m(u,v)$, is the maximum dimension of a subspace $W\subset H_{loc}^1(D)\times H_{loc}^1(D)$ such that $Q_{u,v}(\phi,\psi;D)<0$ for all $(\phi,\psi) \in W\backslash \{0\}$.
\end{defn}

To present the abstract theorem, we first recall a general framework, as outlined in \cite{BCJS2024} (see also \cite{Berestycki-1983,BCJS2023,CJS2024}).
Let $(E, \langle\cdot,\cdot\rangle)$ and $(H,(\cdot, \cdot))$
be two infinite dimensional Hilbert spaces. We assume there exist continuous injections:
$ E\hookrightarrow H \hookrightarrow E'$
such that the injection $E\hookrightarrow H$ has a norm of at most 1. For simplicity, we identify $E$ with its image in $H$. 
For each $u\in E$, we define the norms $\|u\|^2=<u,u>$ and $|u|^2=(u,u)$.
Given $a\in (0, +\infty)$, we define the set $S(a)= \{u\in E, |u|^2= a\}$.
 It is evident that $S(a)$ is a submanifold of $E$.
 We denote by $\|\cdot\|_*$ and $\|\cdot\|_{**}$ the operator norms on $\mathcal{L}(E, R)$ and $\mathcal{L}(E, \mathcal{L}(E, \mathbb{R}))$, respectively.
\begin{defn} \cite{BCJS2024}
	Let $\phi: E\to \mathbb{R}$ be a $C^2$-functional on $E$ and $\alpha\in (0,1]$. We say that $\phi'$ and $\phi''$ are $\alpha$-H\"{o}lder continuous on bounded sets if for any $R>0$ one can find $M=M(R)>0$ such that for any $u_1, u_2\in B(0, R)$:
	\begin{equation*}
		\|\phi'(u_1)-\phi'(u_2)\|_* \leq M\|u_1-u_2\|^\alpha,~~
		\|\phi''(u_1)-\phi''(u_2)\|_{**} \leq M\|u_1-u_2\|^\alpha.
	\end{equation*}
\end{defn}
\begin{defn} \cite{BCJS2024}
	Let $\phi$ be a $C^2$-functional on $E$, for any $u\in E$ define the continuous bilinear map:
	\[D^2\phi(u)=\phi''(u)-\frac{\phi'(u)\cdot u}{|u|^2}(\cdot,	\cdot). \]
\end{defn}
\begin{defn}\cite{BCJS2024}
For any $u\in S(a)$ and $\theta>0$, we define a approximate Morse index by
\begin{align*}
\tilde{m}_\theta(u)=\sup \{dim ~L | &L~ {\rm is~ a~ subspace ~of~ } T_uS(a)~ {\rm such~ that~} \\
&D^2\phi(u)[\varphi,\varphi]<-\theta\|\varphi\|^2, ~\forall \varphi \in L \backslash \{0\}\}.
\end{align*}
\end{defn}
If $u$ is a critical point of the constraint functional $\phi|_{S(a)}$, and if $\theta=0$, then $\tilde{m}_\theta(u)$ is the Morse index of $u$ as a constrained critical point.

\begin{thm}\label{abstract theorem}
	(\cite[Theorem 1.5]{BCJS2024}).
	Let $I\subset(0, +\infty)$ be an interval and consider a family of $C^2$ functionals $\Phi_\rho : E \to \mathbb{R}$ of the form:
	\begin{equation*}
		\Phi_\rho(u)= A(u)-\rho B(u), ~~\rho \in I,
	\end{equation*}
where $B(u)\geq 0$ for every $u\in E$, and
\begin{equation*}
	{\rm either~~ }A(u)\to +\infty ~~{\rm or}~~ B(u)\to +\infty
	{\rm ~~as~~} u\in E ~~{\rm and ~~}\|u\|\to +\infty.
\end{equation*}	
Suppose moreover that $\Phi'_\rho$ and $\Phi''_\rho$ are $\alpha$-H\"{o}lder continuous on bounded sets for some $\alpha \in (0,1]$.
Finally, suppose that there exist $w_1, w_2\in S(a)$ (independent of $\rho$) such that, setting
\begin{equation*}
	\Gamma=\{\gamma\in C([0,1], S(a)): \gamma(0)=w_1, \gamma(1)=w_2\},
\end{equation*}
we have
\begin{equation*}
	c_\rho=\inf_{\gamma\in \Gamma}\max_{t\in [0,1]}\Phi_\rho(\gamma(t))
	>\max\{\Phi_\rho(w_1), \Phi_\rho(w_2)\},~~\forall \rho \in I.
\end{equation*}
Then, for almost every $\rho\in I$, there exist	sequence $\{u_n\}\subset S(a)$ and $\zeta_n\to 0^+$ such that, as $n \to \infty$,
\begin{itemize}
	\item[(i)]  $\Phi_\rho(u_n)\to c_\rho$;
	\item[(ii)] $\|\Phi'_\rho|_{S(a)}(u_n)\|\to 0$;
	\item[(iii)]  $\{u_n\}$ is bounded in $E$;
	\item[(iv)] $\tilde{m}_{\zeta_n}(u)\leq 1$.
\end{itemize}
\end{thm}

We have the following existence result.
\begin{prop}\label{mp critical pt}
Assume $a\in (0, a^*)$. Then, for almost every $\rho\in [\frac{1}{2}, 1]$, there exists a critical point $(u_\rho,v_\rho)\in \mathcal{S}_a$ of $J_\rho|_{\mathcal{S}_a}$ at the mountain pass level $c_{\rho}$, which satisfies 
\begin{eqnarray}\label{approximategrad}
\begin{cases}
-\Delta u_\rho+V_1(x)u_{\rho}+ \lambda_\rho u_\rho= \rho u_\rho v_\rho^2,  & x\in\Omega,\\
-\Delta v_\rho+V_2(x)v_{\rho}+ \lambda_\rho v_\rho= \rho u_\rho^2v_\rho,  & x\in\Omega,\\
\frac{\partial u_\rho}{\partial \nu}=\frac{\partial v_\rho}{\partial \nu}=0,  \, &x\in\partial \Omega,\\
\int_{\Omega} \left(|u_\rho|^2+ |v_\rho|^2\right) dx =a
\end{cases}
\end{eqnarray}
for some $\lambda_\rho \in \mathbb{R}$.
Moreover, $m(u_\rho,v_\rho)\leq 2$.
\end{prop}
\begin{proof} Fix $a\in (0, a^*)$.
 Then Theorem \ref{abstract theorem} can be applied to the family of functionals 
$J_\rho$, with the choices $E=H^1(\Omega)\times H^1(\Omega)$, $H=L^2(\Omega)\times L^2(\Omega)$, $S(a)=\mathcal{S}_a$, and $\Gamma$ is defined by \eqref{Gammagrad}.
Define 
\begin{align*}
	&A(u,v):= \int_{\Omega}\left(|\nabla u|^2+ |\nabla v|^2 \right)dx+\int_{\Omega}\left(V_1(x)|u|^2+ V_2(x)|v|^2 \right)dx,\\
  &B(u,v):=\int_{\Omega}  u^2v^2 dx.
\end{align*}
By Theorem \ref{abstract theorem}, for almost every $\rho\in [\frac{1}{2}, 1]$, we obtain a bounded Palais-Smale sequence $\{(u_n, v_n)\}\subset H^1(\Omega)\times H^1(\Omega)$ for the constrained functional $J_{\rho}|_{\mathcal{S}_a}$ at the level $c_\rho$.
Moreover, there exists $\zeta_n \to 0^+$, such that $\tilde{m}_{\zeta_n}(u_n,v_n)\leq 1$. 
Since $\|J'_\rho |_{\mathcal{S}_a}(u_n,v_n) \|\to 0$, by \cite{Berestycki-1983}, there exists a sequence $\{\lambda_n\}\subset \mathbb{R}$ such that for all $(\varphi,\psi )\in H^1(\Omega)\times H^1(\Omega)$,
\begin{eqnarray}\label{weaksseqgrad}
&\int_{\Omega} \left(\nabla u_n \nabla \varphi+\nabla v_n \nabla \psi\right) dx+\int_{\Omega}\left(V_1(x) u_n \varphi+V_2(x)v_n\psi\right) dx \nonumber\\
 &+ \lambda_n \int_{\Omega}\left( u_n \varphi+v_n\psi\right) dx -\rho\int_{\Omega}\left(u_nv_n^2\varphi+ u_n^2v_n\psi\right) dx
=o_n(1).
\end{eqnarray}
This implies
\begin{align*}
	&\lambda_n +o_n(1)\\
  &= \frac{1}{a}\bigg [-\int_{\Omega}\left(|\nabla u_n|^2+|\nabla v_n|^2 \right)dx-\int_{\Omega}\left(V_1(x)|u_n|^2+V_2(x)|v_n|^2 \right)dx\\
  &\quad+ 2\rho\int_{\Omega}  u_n^2v_n^2 dx \bigg ].
\end{align*}
By the boundedness of $\{u_n\}$ and $\{v_n\}$, we deduce that $\{\lambda_n\}$ is also bounded. 
We may assume, up to subsequences, that $\lambda_n \to \lambda_\rho \in \mathbb{R}$. Furthermore, since $\Omega $ is bounded, by standard arguments, we infer that $(u_n,v_n) \to (u_\rho ,v_\rho)$ in $H^1(\Omega)\times H^1(\Omega)$. 

It remains to show the Morse index $m(u_\rho,v_\rho)\leq 2$.
Since the tangent space $T_{(u,v)}\mathcal{S}_a$ has codimension $1$, in view of $\tilde{m}_{\zeta_n}(u_n,v_n)\leq 1$, also by the fact $d^2|_{\mathcal{S}_a}J_\rho$ and the tangent space $T_{(u_n,v_n)}\mathcal{S}_a$ vary with continuity, taking the limit, we infer that $\tilde{m}_0(u_\rho,v_\rho) \leq 1$, see \cite{BCJS2023}.
Then, using similar arguments as in \cite{BCJS2023} we obtain the desired estimate $m(u_\rho,v_\rho)\leq 2$.
\end{proof}

We now consider the nontriviality of solution $(u_\rho,v_\rho)$ in Proposition \ref{mp critical pt}. 
Regarding the Lagrange multiplier $\lambda_\rho$ associated with the semi-trivial solutions, we have the following property.
\begin{lem}
Assume that $m>0$ and $a\in (0, a^*)$ and let $(u_\rho,v_\rho)$ be a semi-trivial solution to \eqref{approximategrad}. Then 
\begin{equation*}
	\lambda_\rho \leq -\frac{1}{a}m.~~ 
\end{equation*}
\end{lem}
\begin{proof}
	Since $(u_\rho,v_\rho,\lambda_\rho)$ is a solution of \eqref{approximategrad}, it holds that
\begin{equation}
	J_\rho'(u_\rho,v_\rho)(u_\rho,v_\rho) + 2\lambda_\rho a = 0.
\end{equation}
From $J_\rho(u_\rho,v_\rho) - \frac{1}{2}J_\rho'(u_\rho,v_\rho)(u_\rho,v_\rho) = \rho \int_\Omega u_\rho^2v_\rho^2\,dx $, it follows that
\begin{eqnarray*}
	c_\rho + \lambda_\rho a =\rho \int_{\Omega}u_\rho^2v_\rho^2dx.
\end{eqnarray*}
If $(u_\rho,v_\rho)$ is a semi-trivial solution, then $\int_\Omega u_\rho^2v_\rho^2\,dx=0$. 
Hence by Lemma \ref{uni mountain}, we obtain
\begin{equation*}
	\lambda_\rho = -\frac{1}{a}c_\rho \leq -\frac{1}{a}m.
\end{equation*}
The proof is complete.
\end{proof}

On the other hand, the Lagrange multiplier $\lambda_\rho$ admits uniform upper and lower bounds  with respect to the mass $a$.
\begin{lem}\label{bound lambda}
	Assume that $(u_\rho,v_\rho)$ is a semi-trivial solution to  \eqref{approximategrad}. Then there exist constants $\underline{\lambda}\leq \overline{\lambda}$ such that
\begin{equation*}
	\underline{\lambda}\leq \lambda_\rho \leq \overline{\lambda}, ~~\forall \rho\in \hat I.
\end{equation*}
\end{lem}
\begin{proof}
Without loss of generality, we may assume that $v_\rho\equiv 0$. Then system \eqref{approximategrad} reduces to the single equation
\begin{equation}\label{single eq grad}
\begin{cases}
-\Delta u_\rho+ V_1(x)u_\rho=-\lambda_\rho u_\rho	,  & x\in \, \Omega,\\
\frac{\partial u_\rho}{\partial \nu}=0,  \, &x\in\partial \Omega,\\
\int_{\Omega} |u_\rho|^2 dx=a.
\end{cases}
\end{equation}
Consequently, $u_\rho$ is an eigenfunction of the operator $-\Delta + V_1(x)$ with Neumann boundary conditions associated to the eigenvalue $-\lambda_\rho$. From $m(u_\rho,0)\le 2$, it follows that $-\lambda_\rho=\mu_i$ for some $i\in \{0,1,2\}$, where $\mu_i$ denotes the $i$-th eigenvalue of
\begin{equation*}
\begin{cases}
-\Delta u + V_1(x)u = \mu u,  & x\in \, \Omega,\\
\frac{\partial u}{\partial \nu}=0,  \, &x\in\partial \Omega.
\end{cases}
\end{equation*}
Similarly, if $u_\rho\equiv 0$, then $-\lambda_\rho =\tilde\mu_i$ for some $i\in \{0,1,2\}$, where $\tilde \mu_i$ denotes the $i$-th eigenvalue of
\begin{equation*}
\begin{cases}
-\Delta v + V_2(x)v = \mu v,  & x\in \, \Omega,\\
\frac{\partial v}{\partial \nu}=0,  \, &x\in\partial \Omega.
\end{cases}
\end{equation*}
Setting $\underline{\lambda}=\min\{-\mu_2,-\tilde \mu_2\}$ and $\overline{\lambda}=\max\{-\mu_0,-\tilde\mu_0\}$, we conclude that
\begin{equation*}
	\underline{\lambda}\leq \lambda_\rho \leq\overline{\lambda}.
\end{equation*}
This completes the proof.
\end{proof}

\begin{lem}\label{nontrivial sol}
	There exists $\hat a^*\in (0, a^*)$ such that whenever $a\in (0, \hat a^*)$, any solution $(u_\rho,v_\rho,\lambda_\rho)$ to \eqref{approximategrad} must be nontrivial.
\end{lem}
\begin{proof}
	Supposing instead that $(u_\rho, v_\rho)$ is a semi-trivial solution, we then take $a$ sufficiently small such that $a < -\frac{m}{\underline{\lambda}}$. This implies $-\frac{1}{a}m < \underline{\lambda}$, contradicting Lemma \ref{bound lambda}. 
Thus, by defining $\hat a^* := \min\left\{a^*, -\frac{m}{\underline{\lambda}}\right\}$, we conclude the proof.
\end{proof}

\section{Blow-up analysis}\label{blow up ana}

In this section, we aim to establish that the sequence $\{({u_{\rho_n}}, v_{\rho_n})\}$ is bounded in $H^1(\Omega)$. Once this boundedness is confirmed, we will then show that, as $\rho_n\to 1^-$, the sequence $\{({u_{\rho_n}}, v_{\rho_n})\}$ converges to a constrained critical point of $J_1$.

To simplify notation,  we define
 $u_n:= u_{\rho_n}$, $v_n:= v_{\rho_n}$, $\lambda_n:=\lambda_{\rho_n}$, $c_n:= c_{\rho_n}$, and $J_{n}:=J_{\rho_n}$. The pair
 $(u_n,v_n)$ is a weak solution of the following system
\begin{eqnarray}\label{withngrad}
\begin{cases}
-\Delta u_n+V_1(x)u_n+ \lambda_n u_n=\rho_n u_nv_n^2,  & x\in \Omega,\\
-\Delta v_n+V_2(x)v_n+ \lambda_n v_n=\rho_n  u_n^2v_n,  &x\in \Omega,\\
\frac{\partial u_n}{\partial \nu}=\frac{\partial v_n}{\partial \nu}=0,  \, &x\in\partial \Omega,\\
\int_\Omega \left(|u_n|^2+|v_n|^2\right) dx=a.
\end{cases}
\end{eqnarray}
Additionally, it holds that $m(u_n, v_n)\le2$.

From (\ref{withngrad}), we obtain
\begin{eqnarray}\label{1-2-1}
c_n+\lambda_n a =\rho_n \int_{\Omega}u_n^2v_n^2dx.
\end{eqnarray}
If $(u_n, v_n)$ is a semi-nontrivial solution, i.e., $u_n\equiv 0$ or $v_n\equiv 0$, then (\ref{1-2-1}) gives $c_n+\lambda_n a=0$. Since $\rho_n \nearrow 1^-$ and $c_n>0$, taking $n$ sufficiently large, there exist constants $\sigma_1, \sigma_2>0$ independent of $n$ such that $\sigma_1\le c_n \le \sigma_2$, which shows that $\{\lambda_n\}$ is bounded. Then it follows that $\{(u_n, v_n)\}$ is also bounded in $H^1(\Omega)\times H^1(\Omega)$.

In what follows, we assume that $(u_n, v_n)$ is a nontrivial solution.

From (\ref{withngrad}) and the boundedness of $\{c_n\}$, we deduce the following: if the sequence $\{\lambda_n\}$ is bounded, then $\{(u_n, v_n)\}$ is bounded in $H^1(\Omega) \times H^1(\Omega)$. Moreover, it is straightforward to verify that $\lambda_n \ge -C$ for some constant $C > 0$ independent of $n$.  Furthermore, we have the following lemma.
\begin{lem}\label{udgrad}
If $\lambda_n \to +\infty$, then $\max\{\|u_n \|_{L^\infty} ,\|v_n \|_{L^\infty} \}\to +\infty$.
\end{lem}
\begin{proof}
	Since $(u_n,v_n)$ satisfies \eqref{withngrad}, we get that 
\begin{equation*}
		\int_{\Omega}\left(|\nabla u_n|^2+ |\nabla v_n|^2\right)dx+\int_{\Omega}\left(V_1(x)u_n^2+V_2(x)v_n^2\right)dx+\lambda_n a=2\rho_n\int_{\Omega} u_n^2v_n^2dx.
	\end{equation*}
	Thus, we obtain
	\begin{equation*}
		\lambda_n a\leq \left(\max_{i=1,2}\|V_i\|_{L^\infty}\right)a + 2|\Omega|\|u_n \|_{L^\infty}^2\|v_n \|_{L^\infty}^2,
	\end{equation*}
	which implies this lemma and the proof is complete.
\end{proof}

We will perform a blow-up analysis to characterize the asymptotic behavior of the solutions as $\lambda_n\to +\infty$ and in this process, we need the following Liouville type theorems, the proofs of which will be postponed to the appendix.



\begin{prop}\label{wholeliou}
Let $(u,v)$ be a bounded solution of the system:
\begin{eqnarray*}
\begin{cases}
-\Delta u =  uv^2,  & x\in \mathbb{R}^3 ,\\
-\Delta v =  u^2v , & x\in \mathbb{R}^3 .\\
\end{cases}
\end{eqnarray*}
If $(u,v)$ has a finite Morse index, then $u\equiv v \equiv 0$.
\end{prop}
\begin{prop}\label{halfliou}
Let $(u,v)$ be a bounded solution of the  system:
\begin{eqnarray}\label{tdxitong}
\begin{cases}
-\Delta u =  uv^2,  & x\in \mathbb{R}^3_+,\\
-\Delta v =  u^2v,  & x\in \mathbb{R}^3_+,\\
\frac{\partial u}{\partial \nu}=\frac{\partial v}{\partial \nu}=0,  \, &x\in\partial \mathbb{R}^3_+.\\
\end{cases}
\end{eqnarray}
If $(u,v)$ has a finite Morse index, then $u\equiv v \equiv 0$.
\end{prop}

We first give the local behavior of $(u_n, v_n)$ as follows.

\begin{thm}\label{localbegrad}
Suppose that $\lambda_n \to +\infty$.
Let $P_n\in \overline{\Omega}$ be a point such that, for some sequence $R_n\to \infty$,
 \begin{equation*}
 	|u_n(P_n)|=\max_{B_{R_n\tilde{\epsilon}_n}(P_n) \cap \overline{\Omega}}|u_n(x)|~~ \mbox{or}~~ |v_n(P_n)|=\max_{B_{R_n\tilde{\epsilon}_n}(P_n) \cap \overline{\Omega}}|v_n(x)|,
 	 \end{equation*}
where $\tilde{\epsilon}_n=|u_n(P_n)|^{-1} \to 0$ or $\tilde{\epsilon}_n=|v_n(P_n)|^{-1} \to 0$. Set $\epsilon_n=\lambda_n^{-\frac{1}{2}}$. 
Then
  \begin{equation}\label{3.8.1grad}
 \left(\frac{\tilde{\epsilon}_n}{\epsilon_n} \right)^2\to \tilde{\lambda} \in (0, 1].
 \end{equation}
Furthermore, suppose that
 \begin{equation}\label{innergrad}
 	\limsup_{n\to +\infty}\frac{dist(P_n, \partial\Omega)}{\tilde{\epsilon}_n}= +\infty.
 \end{equation}
Then, up to a subsequence, the following results hold:
 \begin{itemize}
	\item[(i)] $P_n\to P\in\Omega$;
    \item[(ii)] 
$\frac{dist(P_n, \partial\Omega)}{\epsilon_n}\to+\infty$ as $n\to+\infty$, and the scaled sequences
 	\begin{equation}\label{equivalentgrad}
 \bar{u}_n(x)= \epsilon_n u_n(\epsilon_n x+ P_n),~~ \bar{v}_n(x)= \epsilon_n v_n(\epsilon_n x+ P_n),~~ x\in \Omega_n:=\frac{\Omega-P_n}{\epsilon_n}
 	\end{equation}
 converge respectively to functions $\bar{u}$ and $\bar{v}$ in $C^2_{loc}(\mathbb{R}^3)$ as $n\to \infty$, where the pair $(\bar{u}, \bar{v})$
 solves the system
\begin{equation*}
\begin{cases}
-\Delta \bar{u}+ \bar{u}= \bar{u}\bar{v}^2~~~  &\text{in}~~ \mathbb{R}^3 ,\\
-\Delta \bar{v}+ \bar{v}=\bar{u}^2\bar{v}~~~  &\text{in}~~ \mathbb{R}^3,\\
\bar{u}(x), \bar{v}(x)\to 0~~\mbox{as}~~|x|\to +\infty;
\end{cases}
\end{equation*}
	\item[(iii)] there exists $\phi_n\in C_0^{\infty}(\Omega)$ with supp$\phi_n \subset B_{R\epsilon_n}(P_n)$ for some $R>0$, such that
	\[Q_{u_n,v_n}(\phi_n,\phi_n;\Omega)<0.\]
 \end{itemize}
Instead of \eqref{innergrad}, we suppose that
  \begin{equation}\label{onboundarygrad}
 	\limsup_{n\to +\infty}\frac{dist(P_n, \partial\Omega)}{\tilde{\epsilon}_n}< +\infty.
 \end{equation}
Then, up to a subsequence, the following results hold:
 \begin{itemize}
	\item[(i)] $P_n\to P\in\partial\Omega$;
 	\item[(ii)]
$\frac{dist(P_n, \partial\Omega)}{\epsilon_n}\to d_0\ge0$ as $n\to+\infty$, and the scaled sequences $\{\bar{u}_n\}$ and $\{\bar{v}_n\}$ defined in \eqref{equivalentgrad}  converge respectively to functions $\bar{u}$ and $\bar{v}$ in $C^2_{loc}(\mathbb{R}^3_+)$ as $n\to \infty$, where $(\bar{u}, \bar{v})$
solves the system
 \begin{equation*}
 \left\{\begin{array}{ll}
-\Delta \bar{u}+ \bar{u}=\bar{u}\bar{v}^2~~&\text{in}~~ \{x_3>-d_0\} ,\\
-\Delta \bar{v}+ \bar{v}=\bar{u}^2\bar{v}~~&\text{in}~~ \{x_3>-d_0\},\\
\frac{\partial \bar{u}}{\partial x_3}= \frac{\partial \bar{v}}{\partial x_3}=0~~~  &\text{on}~~ \{x_3=-d_0\},\\
\bar{u}(x) ,\bar{v}(x)\to 0  ~~\mbox{as}~~|x|\to +\infty;
        \end{array}\right.
        \end{equation*}   
	\item[(iii)] there exists $\phi_n\in C_0^{\infty}(\Omega)$ with supp$\phi_n \subset B_{R\epsilon_n}(P_n)\cap \overline{\Omega}$ for some  $R>0$, such that
	\[Q_{u_n,v_n}(\phi_n,\phi_n;\Omega)<0.\]
	\end{itemize}
\end{thm}
\begin{proof}
Without loss of generality, we may assume that
\begin{equation*}
  \max\limits_{x\in \overline{\Omega}}|u_n(x)|\geq\max\limits_{x\in \overline{\Omega}} |v_n(x)|.
\end{equation*}
Firstly, we take the local maximum point $P_n\in \overline{\Omega}$ such that $|u_n(P_n)|=\max\limits_{x\in \overline{\Omega}}|u_n(x)|$. 
From \eqref{withngrad}, we obtain the inequality $\frac{\lambda_n}{ |u_n(P_n)|^2}\leq 1$, which further implies that
\begin{equation*}
	\frac{\lambda_n}{|u_n(P_n)|^{2}}\to \tilde{\lambda}\in [0, 1].
\end{equation*}
We next prove $\tilde{\lambda}>0$. Define
\begin{equation*}
	\tilde{u}_n(x):=\tilde{\epsilon}_n u_n(\tilde{\epsilon}_nx+ P_n),~~
	\tilde{v}_n(x):=\tilde{\epsilon}_n v_n(\tilde{\epsilon}_nx+ P_n), ~~ x\in \tilde{\Omega}_n:=\frac{\Omega-P_n}{\tilde{\epsilon}_n}.
\end{equation*}
Clearly, $(\tilde{u}_n, \tilde{v}_n)$ satisfies
\begin{equation*}
	\begin{cases}
		-\Delta \tilde{u}_n+V_1(\tilde{\epsilon}_nx+ P_n)\tilde{\epsilon}_n^2 \tilde{u}_n+ \lambda_n\tilde{\epsilon}_n^2 \tilde{u}_n=\rho_n\tilde{u}_n\tilde{v}_n^2,  & x\in \tilde{\Omega}_n,\\
-\Delta \tilde{v}_n+V_2(\tilde{\epsilon}_nx+ P_n)\tilde{\epsilon}_n^2 \tilde{v}_n+ \lambda_n\tilde{\epsilon}_n^2 \tilde{v}_n=\rho_n  \tilde{u}_n^2\tilde{v}_n,  & x\in \tilde{\Omega}_n,\\
\frac{\partial \tilde{u}_n}{\partial \nu}=\frac{\partial \tilde{v}_n}{\partial \nu}=0,  \, &x\in\partial \tilde{\Omega}_n,\\
\max\limits_{x\in \overline{\tilde{\Omega}}_n}|\tilde{v}_n|\leq\max\limits_{x\in \overline{\tilde{\Omega}}_n}|\tilde{u}_n|= |\tilde{u}_n(0)|= 1.
	\end{cases}
\end{equation*}
Let $d_n:=dist(P_n, \partial \Omega)$.
Then
\begin{equation*}
	\frac{d_n}{\tilde{\epsilon}_n}\to L\in [0, +\infty]
	~~\mbox{and}~~\tilde{\Omega}_n\to 
	\begin{cases}
	\mathbb{R}^3,~~&L=+\infty;\\
	\mathbb{H},~~&L<+\infty,	
	\end{cases}
\end{equation*}
where $\mathbb{H}$ denotes a half-space such that $0\in \overline{\mathbb{H}}$ and $d(0, \partial \mathbb{H})=L$.

By the elliptic regularity theory it follows that, up to a subsequence, $\tilde{u}_n\to \tilde{u}$ and $\tilde{v}_n\to \tilde{v}$ in $C^2_{loc}(\overline{D})$, where $(\tilde{u}, \tilde{v})$ solves
\begin{equation*}
\begin{cases}
	-\Delta \tilde{u}+ \tilde{\lambda}\tilde{u}=\tilde{u}\tilde{v}^2,  & x\in D,\\
-\Delta \tilde{v}+ \tilde{\lambda} \tilde{v}=\tilde{u}^2\tilde{v},  & x\in D,\\
\frac{\partial \tilde{u}}{\partial \nu}=\frac{\partial \tilde{v}}{\partial \nu}=0,  \, &x\in\partial D,
\end{cases}
\end{equation*}
where $D$ is either $\mathbb{R}^3$ or $\mathbb{H}$.

Now we claim that
\begin{equation}\label{aamwi}
 \text{the Morse index of}~	(\tilde{u}, \tilde{v}) ~\text{is less than or equal to 2}.
\end{equation}
In fact, if not, we may assume that there exist $k>2$ functions $\phi_1,\cdots, \phi_k\in C_0^{\infty}(D)$ and $\phi_1,\cdots, \phi_k$ are orthogonal in $H^1(D)$, such that
\[
 2\int_{D} |\nabla \phi_i|^2 dx +2\tilde{\lambda}\int_{D}|\phi_i|^2dx -\int_{D}\left(\tilde{v}^2\phi_i^2+\tilde{u}^2\phi_i^2+4\tilde{u}\tilde{v}\phi_i^2\right) dx<0
\]
holds for all $i\in \{1,\cdots, k\}$
 Moreover, let \[\phi_{i,n}(x)=\tilde{\epsilon}_n^{-\frac{1}{2}}\phi_i (\frac{x-P_n}{\tilde{\epsilon}_n}).\]
By a direct computation, we have
\begin{align}\label{localbu eq1}
	&Q_{u_n,v_n}(\phi_{i,n},\phi_{i,n};\Omega)\nonumber\\
  &= 2\int_{\Omega}\left(|\nabla \phi_{i,n} |^2 +\lambda_n|\phi_{i,n} |^2\right) dx+\int_{\Omega} (V_1(x)+V_2(x))|\phi_{i,n}|^2 dx\nonumber\\
	&\quad- \rho_n \int_{\Omega} \left(v_n^2\phi_{i,n}^2+ u_n^2\phi_{i,n}^2+ 4 u_nv_n\phi_{i,n}^2 \right)dx\nonumber\\
	&=2\int_{\tilde{\Omega}_n}\left(|\nabla \phi_i|^2 +\lambda_n\tilde{\epsilon}_n^2 |\phi_i|^2\right)dx\nonumber\\
  &\quad+\int_{\tilde{\Omega}_n} (V_1(\tilde{\epsilon}_nx+P_n)+V_2(\tilde{\epsilon}_nx+P_n))\tilde{\epsilon}_n^2|\phi_{i}|^2dx\nonumber \\
	&\quad-\rho_n\int_{\tilde{\Omega}_n} \left( \tilde{v}_n^2\phi_i^2+\tilde{u}_n^2\phi_i^2+4\tilde{u}_n\tilde{v}_n\phi_i^2 \right)dx\nonumber\\
	&\to 2\int_{D} |\nabla \phi_i|^2 dx +2\tilde{\lambda}\int_{D}|\phi_i|^2dx -\int_{D}\left(\tilde{v}^2\phi_i^2+\tilde{u}^2\phi_i^2+4\tilde{u}\tilde{v}\phi_i^2\right) dx<0,
\end{align}
which contradicts to $m(u_n, v_n)\leq 2$. Hence \eqref{aamwi} holds.

Having established that $(\tilde{u}, \tilde{v})$ has finite Morse index, applying the Liouville type results, see Proposition \ref{wholeliou} and Proposition \ref{halfliou},
we conclude that the occurrence of $\tilde{\lambda}=0$ is ruled out, regardless of $D$ is a half-space or a whole space.
Hence, $\tilde{\lambda}\in (0,1]$. 	

We note that 
$(\bar{u}_n, \bar{v}_n)$ defined in \eqref{equivalentgrad} satisfies
\begin{equation*}
	\begin{cases}
-\Delta \bar{u}_n+ \bar{u}_n=\rho_n  \bar{u}_n\bar{v}_n^2,  & x\in \Omega_n,\\
-\Delta \bar{v}_n+ \bar{v}_n=\rho_n \bar{u}_n^2\bar{v}_n,  & x\in\Omega_n,\\
\frac{\partial \bar{u}_n}{\partial \nu}=\frac{\partial \bar{v}_n}{\partial \nu}=0,  \, &x\in\partial \Omega_n,\\
\max\limits_{x\in \overline{\Omega}_n} |\bar{v}_n|\leq \max\limits_{x\in \overline{\Omega}_n}|\bar{u}_n|= \frac{\epsilon_n}{\tilde{\epsilon}_n}\to \tilde{\lambda}^{-\frac{1}{2}}.
	\end{cases}
\end{equation*}
Up to a subsequence, we have  $\bar{u}_n \to \bar{u}$ and $\bar{v}_n \to \bar{v}$ in $C^2_{loc}(\overline{D})$, where $D$ is either $\mathbb{R}^3$ or a half-space $\mathbb{H}$, and $(\bar{u}, \bar{v})$ solves
\begin{equation}\label{jinshenghgrad}
	\begin{cases}
-\Delta \bar{u}+ \bar{u}=\bar{u}\bar{v}^2,  & x\in D,\\
-\Delta \bar{v}+ \bar{v}=\bar{u}^2\bar{v},  & x\in D,\\
\frac{\partial \bar{u}}{\partial \nu}=\frac{\partial \bar{v}}{\partial \nu}=0,  \, &x\in\partial D,\\
\max\limits_{x\in \overline{D}} |\bar{v}|\leq \max\limits_{x\in\overline{D}}|\bar{u}|= \tilde{\lambda}^{-\frac{1}{2}}.
	\end{cases}
\end{equation}
Furthermore, we also have $m(\bar u, \bar v)\leq 2$. 
Assuming now (up to a subsequence) that $P_n\to P\in \overline{\Omega}$, we consider the following two cases:
\begin{itemize}
	\item[(a)]  If $\limsup\limits_{n\to +\infty}\frac{dist(P_n, \partial\Omega)}{\tilde{\epsilon}_n}= +\infty$, then $\limsup\limits_{n\to +\infty}\frac{dist(P_n, \partial\Omega)}{\epsilon_n}= +\infty$ and thus $D=\mathbb{R}^3$. Since $m(\bar u, \bar v)\le 2$, standard regularity arguments imply that $\bar{u},\bar{v}\in C^2(\mathbb{R}^3)$ and $\bar{u}(x),\bar{v}(x)\to 0$ as $|x|\to \infty$.

	 \item[(b)] If $0\le\limsup\limits_{n\to +\infty}\frac{dist(P_n, \partial\Omega)}{\tilde{\epsilon}_n}< +\infty$, then $0\le\limsup\limits_{n\to +\infty}\frac{dist(P_n, \partial\Omega)}{\epsilon_n}< +\infty$
and we may assume that $\frac{dist(P_n, \partial\Omega)}{\epsilon_n}\to d_0\ge 0$ as $n\to +\infty$. Up to a space rotation, we may assume that the tangent to $\partial \Omega$ at $P$ is parallel to the plane $\{x_3=-d_0\}$.
Then, up to subsequences,  $\bar{u}_n\to \bar{u}$ and $\bar{v}_n\to \bar{v}$ in $ C^2_{loc}(\{x_3>-d_0\})$, and  $(\bar{u},\bar{v})$ satisfies                      
\begin{equation*}
\begin{cases}
	-\Delta \bar{u}+ \bar{u}= \bar{u}\bar{v}^2~  &\mbox{in}~~ \{x_3>-d_0\},\\
-\Delta \bar{v}+ \bar{v}= \bar{u}^2\bar{v}~  &\mbox{in}~~ \{x_3>-d_0\},\\
|\bar v(x)|\leq|\bar u(x)|\leq|\bar u(0)|=~\tilde{\lambda}^{-\frac{1}{2}}~  &\mbox{in}~~ \{x_3>-d_0\},\\
\frac{\partial \bar{u}}{\partial x_3}=\frac{\partial \bar{v}}{\partial x_3}=0~ &\mbox{on}~~ \{x_3=-d_0\}. 
\end{cases}
\end{equation*}     
We extend $\bar{u}$ and $\bar{v}$ by reflection with respect to $\{x_3= -d_0\}$ as follows:
        \begin{equation*}
        \bar{u}_*(x_1,x_2,x_3):=
        	\begin{cases}
        		\bar{u}(x_1,x_2,x_3-d_0)~~&\text{if}~~x_3\geq 0,\\
        		\bar{u}(x_1,x_2,-x_3-d_0)~~&\text{if}~~x_3<0,
        		\end{cases}
        \end{equation*}
        \begin{equation*}
        	\bar{v}_*(x_1,x_2,x_3):=
        	\begin{cases}
        		\bar{v}(x_1,x_2,x_3-d_0)~~&\text{if}~~x_3\geq 0,\\
        		\bar{v}(x_1,x_2,-x_3-d_0)~~&\text{if}~~x_3<0.
        	\end{cases}
        \end{equation*}
The pair $(\bar{u}_*,\bar{v}_*)$ then solves the system:
\begin{equation}\label{hebijigrad}
\left\{
\begin{aligned}
& -\Delta \bar{u}_* + \bar{u}_* = \bar{u}_*\bar{v}_*^2 && \text{in } \mathbb{R}^3,\\
& -\Delta \bar{v}_* + \bar{v}_* = \bar{u}_*^2\bar{v}_* && \text{in } \mathbb{R}^3,\\
& |\bar v_*(x)| \leq |\bar u_*(x)| \leq |\bar u_*(0,0,d_0)| \\
& \qquad = |\bar u_*(0,0,-d_0)| = \tilde{\lambda}^{-\frac{1}{2}} && \text{in } \mathbb{R}^3,\\
& \frac{\partial \bar{u}}{\partial x_3} = \frac{\partial \bar{v}}{\partial x_3} = 0 && \text{on } \{x_3=0\}.
\end{aligned}
\right.
\end{equation}
It is evident that system \eqref{hebijigrad} can be solved for all $d_0\ge0$. Thus, by applying the elliptic regularity arguments and using the finite Morse index of the solution $(\bar{u}_*,\bar{v}_*)$, we deduce that $\bar{u}_*,\bar{v}_*\in C^2(\mathbb{R}^3)$ and $\bar{u}_*(x),\bar{v}_*(x)\to 0$ as $|x|\to \infty$.
\end{itemize}

Finally, by an argument similar to the one used for \eqref{localbu eq1} (see also \cite{EP2011,CRZ2025}), one readily verifies (iii), so the proof is complete.
\end{proof}

\begin{rem}\label{property for sol}{
We note that the system (\ref{jinshenghgrad}) is solvable in $H^1(D)\times H^1(D)$ for $D=\mathbb{R}^3$ and $D=\mathbb{R}_+^3$. In particular, the positive solution $(\bar u, \bar v)=(U_0, U_0)$ exists for $D=\mathbb{R}^3$, where $U_0$ is the unique radial ground-state of $-\Delta u+u=u^{3}$.  }
\end{rem}

In the following, we demonstrate the decay of the sum $u_n+v_n$ away from the potential blow-up points, thereby deriving the global asymptotic behavior of the pair $(u_n, v_n)$ on $\overline{\Omega}$. 

\begin{thm}\label{globalbegrad}
Assume that $\lambda_n\to +\infty$. Then
	there exist $k\in \{1,2\}$ and sequences of points $\{P_n^1\},\cdots,\{P_n^k\}$, such that
	\begin{equation}\label{faraway1grad}
		\lambda_n |P_n^i-P_n^j|^2\to +\infty, ~~\forall i\neq j,~~n\to \infty,
	\end{equation}
	\begin{align}\label{localmaxigrad}
		&\text{either}~|u_n(P_n^i)|= \max_{B_{R_n\lambda_n^{-\frac{1}{2}}}(P_n^i)\cap \overline{\Omega}}|u_n|\nonumber\\
		&\text{or}~|v_n(P_n^i)|= \max_{B_{R_n\lambda_n^{-\frac{1}{2}}}(P_n^i)\cap \overline{\Omega}}|v_n|,~~ \mbox{for~some~} R_n\to \infty, ~\mbox{for~every~} i,
		\end{align}
and constants $C_1>0$, $C_2>0$ such that
	\begin{align}\label{globalbegrad eq}
		&|u_n(x)|, |v_n(x)|\leq C_1 e^{C_1R}\lambda_n^{\frac{1}{2}}\sum_{i=1}^k e^{-C_2\lambda_n^{\frac{1}{2}}|x-P_n^i|},\nonumber\\ 
    &\forall x\in \overline{\Omega}\backslash \cup_{i=1}^k (B_{R\lambda_n^{-\frac{1}{2}}}(P_n^i)\cap \overline{\Omega}).
	\end{align}
\end{thm}

\begin{proof} Following \cite{CRZ2025}(see also \cite{CJS2024,EP2011}), we present the proof in two steps.

{\bf Step 1.}
	There exist $k\in \{1,2\}$ and sequences of points $\{P_n^1\},\cdots,\{P_n^k\}$ such that  \eqref{faraway1grad} and \eqref{localmaxigrad} hold, and moreover
    \begin{equation}\label{decaygrad}
    	\lim_{R\to \infty}\left(\limsup_{n\to +\infty} \lambda_n^{-\frac{1}{2}} \max\left\{\max_{d_n(x)\geq R\lambda_n^{-\frac{1}{2}}}|u_n(x)|,\max_{d_n(x)\geq R\lambda_n^{-\frac{1}{2}}}|v_n(x)|\right\} \right)=0,
    \end{equation}
where $d_n(x):=\min\{|x-P_n^i|:i=1,\cdots,k\}$ is the distance function from $\{P_n^1,\cdots, P_n^k\}$ for $x\in\overline{\Omega}$.
 
Take $\{P_n^1\}\in \overline{\Omega}$ such that $|u_n(P_n^1)|=\max\limits_{\overline{\Omega}}|u_n(x)|\geq \max\limits_{\overline{\Omega}}|v_n(x)|$.
By contradiction, we may assume that there exists $\delta>0$, such that 
    \begin{equation*}
    	\lim_{R\to \infty}\left(\limsup_{n\to +\infty} \lambda_n^{-\frac{1}{2}} \max\left\{\max_{d_n(x)\geq R\lambda_n^{-\frac{1}{2}}}|u_n(x)|,\max_{d_n(x)\geq R\lambda_n^{-\frac{1}{2}}}|v_n(x)| \right\}\right)\geq 4\delta.
    \end{equation*}
Then for sufficiently large $R>0$, we have 
\begin{equation}\label{pn2grad}
	 \lambda_n^{-\frac{1}{2}} \max\left\{\max_{|x-P_n^1|\geq R\lambda_n^{-\frac{1}{2}}}|u_n(x)|,\max_{|x-P_n^1|\geq R\lambda_n^{-\frac{1}{2}}}|v_n(x)|\right\}\geq 2 \delta. 
\end{equation} 
Let $P_n^2\in \overline{\Omega}\backslash B_{R\lambda_n^{-\frac{1}{2}}}(P_n^1)$ be such that
\begin{equation*}
	\max\left\{|u_n(P_n^2)|,|v_n(P_n^2)|\right\}\ge\max\left\{\max_{\overline{\Omega}\backslash B_{R\lambda_n^{-\frac{1}{2}}}(P_n^1)}|u_n|,\max_{\overline{\Omega}\backslash B_{R\lambda_n^{-\frac{1}{2}}}(P_n^1)}|v_n|\right\}.
\end{equation*}
Without loss of generality, we may assume that 
 \begin{equation*}
 	|u_n(P_n^2)|=\max_{\overline{\Omega}\backslash B_{R\lambda_n^{-\frac{1}{2}}}(P_n^1)}|u_n|\geq \max_{\overline{\Omega}\backslash B_{R\lambda_n^{-\frac{1}{2}}}(P_n^1)}|v_n|.
 \end{equation*}
By \eqref{pn2grad}, we have $|u_n(P_n^2)|\to +\infty$ as $n\to +\infty$.

We now prove
 \begin{equation}\label{farawaygrad}
 	\lambda_n^{\frac{1}{2}}|P_n^1-P_n^2|\to +\infty.
 \end{equation}
If \eqref{farawaygrad} does not holds, we may assume that
 $\lambda_n^{\frac{1}{2}}|P_n^1-P_n^2|\to R'\geq R$.
In view of the arguments for Theorem \ref{localbegrad}, we have
 \begin{equation*}
 	\lambda_n^{-\frac{1}{2}}u_n(\lambda_n^{-\frac{1}{2}}x + P_n^1):=\bar{u}_n^1(x)\to \bar{u}(x) ~\mbox{in}~C^2_{loc}(D),
 \end{equation*}
 where $D$ is either $\mathbb{R}^N$ or a half-space.
Then, up to subsequences,
 \begin{equation*}
 	\lambda_n^{-\frac{1}{2}}u_n(P_n^2)=\bar{u}_n^1(\lambda^{\frac{1}{2}}(P_n^2-P_n^1)) \to \bar{u}(x'),~~~|x'|\geq R'>R.
 \end{equation*}
Since $\bar u(x)\to 0$ as $|x|\to +\infty$, we can choose $R$ larger such that
  \begin{equation*}
  	|\bar{u}(x)|\leq \delta, ~\forall |x|\geq R.
  \end{equation*}
This contradicts to \eqref{pn2grad}. Then \eqref{farawaygrad} is proved.

In what follows, we show
  \begin{equation}\label{yanggrad}
  	|u_n(P_n^2)|=\max_{B_{R_n^{(2)}\lambda_n^{-\frac{1}{2}}}(P_n^2)\cap \overline{\Omega}}|u_n| ~~~\mbox{for some}~ R_n^{(2)}\to +\infty.
  \end{equation}
  Let $\tilde{\epsilon}_n^{(2)}=|u_n(P_n^2)|^{-1}$. Clearly $\tilde{\epsilon}_n^{(2)}\to 0$. By \eqref{pn2grad},  we have $\tilde{\epsilon}_n^{(2)}\leq (2\delta)^{-1}\lambda_n^{-\frac{1}{2}}$.
It follows from \eqref{farawaygrad} that 
 \begin{equation*}
 	\tilde{R}_n^{(2)}:=\frac{|P_n^1-P_n^2|}{2\tilde{\epsilon}_n^{(2)}}\geq \frac{2\delta}{2}\lambda_n^{\frac{1}{2}}|P_n^1-P_n^2| \to +\infty,~~n\to\infty.
 \end{equation*}
Note that, for any $x\in B_{\tilde{R}_n^{(2)}\tilde{\epsilon}_n^{(2)}}(P_n^2)$ and any $R>0$, we have
 \begin{equation*}
 |x-P_n^1|\geq |P_n^2-P_n^1|- |x-P_n^2|\geq \frac{1}{2}|P_n^2-P_n^1|\geq R\lambda_n^{-\frac{1}{2}},
 \end{equation*}
which implies that
 \begin{equation*}
 	\overline{\Omega}\cap B_{\tilde{R}_n^{(2)}\tilde{\epsilon}_n^{(2)}}(P_n^2)\subset \overline{\Omega}\backslash B_{R\lambda_n^{-\frac{1}{2}}}(P_n^1).
 \end{equation*}
Hence, taking $R_n^{(2)}= \tilde{R}_n^{(2)}\lambda_n^{\frac{1}{2}}\tilde{\epsilon}_n^{(2)}\to +\infty$, we get \eqref{yanggrad} holds.

If the conclusion (\ref{decaygrad}) does not holds, we can continue to iterate the above arguments for other local maximum and this iteration process must terminate after at most 2 steps, thereby establishing that (\ref{decaygrad}) must eventually hold.
 
 {\bf Step 2.}
Decay estimates of $u_n$ and $v_n$ on $\{x\in \overline{\Omega}: d_n(x)\geq R\lambda_n^{-\frac{1}{2}}\}$. 

For a given $\theta>0$, we define 
$$\Omega_\theta:= \{x\in \overline{\Omega}: {\rm dist} (x,\partial\Omega)<\theta\}, 
$$ 
$$(\partial\Omega)_\theta: =\{(x,y):x\in \partial\Omega, y\in (-\theta,0]\nu_x\}, $$
where $\nu_x$ denotes the unit outer normal of $\partial\Omega$ at $x$. 
Noting that $\partial\Omega$ is a smooth, compact submanifold of $\mathbb{R}^N$, we can invoke the tubular neighborhood theorem to establish the existence of a diffeomorphism $\Phi_{Ib}$ that maps $\Omega_\theta$ onto $(\partial\Omega)_\theta$.
More precisely,  for any point  $x\in \Omega_\theta$,  it is straightforward to see that there exists a unique point $\bar{x}\in \partial\Omega$ such that ${\rm dist}(x, \bar{x})= {\rm dist}(x, \partial\Omega)$. Consequently, we can define $\Phi_{Ib}(x)=(\bar{x}, -{\rm dist}(x, \bar{x})\nu_{\bar{x}})$ for any $x\in \Omega_\theta$.
Furthermore, it is clear that $\Phi_{Ib}(x)=x$ for $x\in \partial\Omega$.

We also define 
$$\Omega^\theta = \{x\in \mathbb{R}^N\backslash \Omega: {\rm dist}(x,\partial\Omega)<\theta\},$$
and 
$$(\partial\Omega)^\theta=\{(x,y):x\in \partial\Omega, y\in [0,\theta)\nu_x\}. $$
Take $\hat{x}\in \partial\Omega$ to be the unique point satisfying ${\rm dist}(x, \hat{x})= {\rm dist}(x, \partial\Omega)$. We can then define a diffeomorphism $\Phi_{Ob}:\Omega^\theta \to (\partial\Omega)^\theta$ by 
\[\Phi_{Ob}(x)=(\hat{x}, dist(x, \hat{x})\nu_{\hat{x}}),~~\forall x\in \Omega^\theta.
\] 

Next, in order to extend system (\ref{grad}) to the larger domain $\overline{\Omega}\cup \Omega^\theta$, we introduce the reflection mapping
\[
\Phi_C:(\partial\Omega)_\theta\to (\partial\Omega)^\theta, ~~\Phi_C((x,y))=(x,-y).
\] 
Finally, we set  \[
\Phi:=\Phi_{Ib}^{-1}\circ \Phi_C^{-1}\circ \Phi_{Ob}.
\]
Then $\Phi$ is a diffeomorphism from $\Omega^\theta$ onto $\Omega_\theta$ and coincides with the identity on $\partial \Omega$.

Furthermore, we set
\begin{eqnarray*}
&&x=\Phi(z)=(\Phi_1(z),\Phi_2(z),\Phi_3(z)), \forall z\in\Omega^\theta, \\
&&z=\Psi(x)=\Phi^{-1}(x)=(\Psi_1(x),\Psi_2(x),\Psi_3(x)),  \forall x\in \Omega_\theta,
\end{eqnarray*}
and define the metric coefficients
\begin{equation*}
	g_{ij}= \sum_{l=1}^3 \frac{\partial \Phi_l}{\partial z_i}\frac{\partial \Phi_l}{\partial z_j},~~
	g^{ij}= \sum_{l=1}^3 \frac{\partial \Psi_i}{\partial x_l}\frac{\partial \Psi_j}{\partial x_l}(\Phi (z)).
\end{equation*}
On the boundary, we have $g_{ij}|_{\partial \Omega}= g^{ij}|_{\partial\Omega}=\delta_{ij}$, where $\delta_{ij}$ denotes the Kronecker symbol.
For simplicity, let $G=(g^{ij})$, $g={\rm det}(g_{ij})$, and define $\hat{u}_n(x)=u_n(\Phi(x))$, $\hat{v}_n(x)=v_n(\Phi(x))$ for $x\in \Omega^\theta$.
The pair $(\hat{u}_n, \hat{v}_n)$ then satisfies
\begin{equation*}
\begin{cases}
	-L \hat{u}_n+\sqrt{g}V_1(\Phi(x)) \hat{u}_n+ \sqrt{g}\lambda_n \hat{u}_n=\rho_n\sqrt{g}\hat{u}_n\hat{v}_n^2, & x\in \Omega^\theta,\\
-L \hat{v}_n+\sqrt{g}V_2(\Phi(x)) \hat{v}_n+ \sqrt{g}\lambda_n \hat{v}_n=\rho_n\sqrt{g}\hat{v}_n\hat{u}_n^2, & x\in \Omega^\theta,\\
\frac{\partial \hat{u}_n}{\partial \nu}=\frac{\partial \hat{v}_n}{\partial \nu}=0, \, &x\in\partial \Omega,
\end{cases}
\end{equation*}
where the operator $L$ is given by:
\[L u:= \sum_{i=1}^3 \frac{\partial}{\partial x_i}\left(\sqrt{g}\sum_{j=1}^3g^{ij}\frac{\partial u}{\partial x_j}\right). \]
We now introduce the piecewise-defined functions
\begin{equation*}
\overline{U}_n(x)=
\begin{cases}
u_n(x), x\in\Omega,\\
\hat{u}_n(x), x\in \Omega^\theta,		
\end{cases}
\overline{V}_n(x)=
\begin{cases}
v_n(x), x\in\Omega,\\
\hat{v}_n(x), x\in \Omega^\theta,		
\end{cases}
\end{equation*}
and the extended metric coefficients
\begin{equation*}
\bar{g}_{ij}=
\begin{cases}
	\delta_{ij}, x\in \Omega,\\
	g_{ij}, x\in \Omega^\theta,
\end{cases}
\bar{g}^{ij}=
\begin{cases}
	\delta_{ij}, x\in \Omega,\\
	g^{ij}, x\in \Omega^\theta.
\end{cases}
\end{equation*}
We also extend the potentials by
\begin{equation*}
\tilde{V}_1(x)=
\begin{cases}
V_1(x), &x\in\Omega,\\
V_1(\Phi(x)), &x\in \Omega^\theta,		
\end{cases}
~~~~~~\tilde{V}_2(x)=
\begin{cases}
V_2(x), &x\in\Omega,\\
V_2(\Phi(x)), &x\in \Omega^\theta,		
\end{cases}
\end{equation*}
Clearly, both the functions $\tilde{V}_i(x), i=1,2$ are bounded in $\overline{\Omega}\cup \Omega^\theta$.
Let $\bar{g}= det(\bar{g}_{ij})$ and 
 define the vector field
 $$A(x, \xi)=(A_1(x, \xi),A_2(x, \xi),A_3(x, \xi)), ~~\xi=(\xi_1,\xi_2,\xi_3),
 $$ 
 where
\[A_i(x,\xi)=\sqrt{\bar{g}}\sum_{j=1}^3\bar{g}^{ij}\xi_j.\]
The pair  $(\overline U_n, \overline V_n)$ then weakly solves the extended system:
\begin{equation*}
\begin{cases}
	-div (A(x, \nabla \overline U_n))+\sqrt{\bar g}\tilde{V}_1(x) \overline U_n+ \sqrt{\bar g}\lambda_n \overline U_n=\rho_n\sqrt{\bar{g}}\overline U_n\overline V_n^2, & x\in \overline{\Omega}\cup \Omega^\theta,\\
-div (A(x, \nabla \overline V_n))+\sqrt{\bar g}\tilde{V}_2(x) \overline V_n+ \sqrt{\bar g}\lambda_n \overline V_n=\rho_n\sqrt{\bar{g}} \overline V_n\overline U_n^2, & x\in \overline{\Omega}\cup \Omega^\theta.
\end{cases}
\end{equation*}
For an arbitrary point $\tilde{x}\in \overline{\Omega}$, choose $r>0$ such that $B_r(\tilde{x})\subset \overline{\Omega}\cup \Omega^\theta$. For $x\in B_r(\tilde{x})\backslash \tilde{x}$, set $\sigma= |x-\tilde{x}|$. If $\phi(\sigma)$ is any smooth increasing function,  a computation analogous to that in  \cite[Lemma 4.5]{CRZ2025} gives
\begin{equation}\label{rongyao}
	\begin{split}
		\left|\sum_{i=1}^3 \frac{\partial}{\partial x_i}\left(\sqrt{g}\sum_{j=1}^3g^{ij}\frac{\partial \phi}{\partial x_j}\right)\right|
	&\leq 2 |\phi''| + \frac{4}{\sigma}\phi' + \frac{C_{\Phi ,\Omega}\phi'}{\sigma},
	\end{split}
\end{equation}
where $\phi'=\frac{d\phi(\sigma)}{d\sigma}$ and $C_{\Phi ,\Omega}>0$ is a constant depending only on $\Phi$ and $\Omega$. Henceforth we fix
 $\theta_*\in (0, \theta)$ such that $\frac{1}{2}\leq\sqrt{\bar{g}}\leq \frac{3}{2}$ and \eqref{rongyao} is satisfied.
 
Define the set 
\begin{equation*}
	\Lambda_n^*:=\{x\in \overline{\Omega}\cup \Omega^{\theta_*}: d_n(x)\geq R\lambda_n^{-1/2}\}.
\end{equation*}
 We shall establish decay estimates for $\{u_n\}$ inside $\Lambda_n^*$; the arguments for
$\{v_n\}$ are completely analogous.  Set $\Lambda_n: =\{x\in \overline{\Omega}: d_n(x)\geq R\lambda_n^{-1/2}\}$. 
In view of \eqref{decaygrad}, for any prescribed $\epsilon\in (0,1)$ (to be determined later), we can find $R^*>0$ and a large positive integer $n_R\in \mathbb{N}$ such that
\begin{equation*}\label{feng}
	\max_{x\in \Lambda_n}|u_n(x)|\leq \lambda_n^{\frac{1}{2}}\epsilon,~~~~~\forall R> R^*, ~~~ \forall n\geq n_R.
\end{equation*}
Moreover, there exists $n_\theta>0$ with the property that for every
 $n>n_\theta$ and every $i\in \{1,\cdots, k\}$, 
 $$B_{\lambda_n^{-1/2}R}(P_n^i)\subset \Omega^{\theta_*}\cup \Omega_{\theta_*}.
 $$
 From the definition of $\overline U_n$, we therefore obtain
 \begin{equation*}
	|\overline U_n(x)|^2\leq \lambda_n \epsilon^2,~~~~~\forall x\in \Lambda_n^*,~\forall n> n_{\theta R}:=\max\{n_\theta, n_R\}. 
 \end{equation*}
Because $\tilde{V}_1(\cdot)$ is bounded, it follows that for sufficiently small $\epsilon>0$
 \begin{equation}\label{lasa}
 -div\left(A(x,\nabla \overline U_n)\right)+ \frac{\lambda_n}{2}\overline U_n\leq 0, ~\forall x\in \Lambda_n^*.
 \end{equation}

Now take an arbitrary $x_0\in {\rm int}(\Lambda_n)$ and set $r_0:={\rm dist}(x_0, \partial \Lambda_n^*)$. Then $B_r(x_0)\subset \Lambda_n^*$. Define the radial function
\begin{equation*}
	\phi_n(|x-x_0|)=\lambda_n^{\frac{1}{2}}\frac{\cosh{(\gamma \lambda_n^{\frac{1}{2}}|x-x_0|})}{\cosh{(\gamma \lambda_n^{\frac{1}{2}}r_0})}, ~~\gamma>0.
\end{equation*}
Clearly,  $\phi'_n(\sigma)>0$ and $\phi''_n(\sigma)>0$ for $\sigma>0$. A direct calculation shows that
\begin{equation*}\label{chichi}
		2|\phi_n''|+\frac{4}{\sigma}\phi_n'+ \frac{C_{\Phi,\Omega}}{\sigma}\phi_n'- \frac{\lambda_n}{2}\phi_n\le 0,~~\forall \gamma\in (0, \gamma_*],
\end{equation*}
where $\gamma_*:=(12+2C_{\Phi,\Omega})^{-\frac{1}{2}}$. Define $\tilde \phi_n(x):=\phi_n(|x-x_0|)$. 
Fixing any $\gamma\in (0, \gamma_*]$, and using \eqref{rongyao} together with the inequality above, we obtain 
\begin{equation}\label{baile}
	\begin{split}
		div(A(x, \nabla \tilde\phi_n))-\frac{\lambda_n}{2}\tilde \phi_n
		\leq 0, ~\forall x\in B_r(x_0).
	\end{split}
\end{equation}
Combining  (\ref{lasa}) with (\ref{baile}) and applying the comparison principle (\cite[Theorem 10.1]{PS-2004}), we conclude 
\begin{eqnarray*}
\overline U_n\leq \tilde \phi_n~~\mbox{in}~~B_r(x_0),
\end{eqnarray*}
which immediately gives 
$$
|u_n(x_0)|\leq \lambda_n^{\frac{1}{2}} e^{-\gamma\lambda_n^{\frac{1}{2}}r}.
$$ 
Two possibilities arise for the distance $r_0={\rm dist}(x_0, \partial \Lambda_n^*)$:
\begin{itemize}
    \item [(i)] $r_0= dist(x_0, B_{R\lambda_n^{-1/2}}(P_n^i))$ for some $i\in \{1,\cdots,k\}$;
	\item [(ii)] $r_0= dist(x_0, \partial(\overline{\Omega}\cup \Omega^{\theta_*}))$.
\end{itemize}
In case (i), we have $|x_0- P_n^i|= r+ R\lambda_n^{-\frac{1}{2}}$, whence 
\[|u_n(x_0)|\leq e^{\gamma R}\lambda_n^{\frac{1}{2}} e^{-\gamma \lambda_n^{\frac{1}{2}}|x_0-P_n^i|}.\]
In case (ii), using the bound $|x_0-P_n^i|\le {\rm diam} (\Omega)$, we deduce 
\[|u_n(x_0)|\leq \lambda_n^{\frac{1}{2}} e^{-\gamma \lambda_n^{\frac{1}{2}}\theta_*}\leq \lambda_n^{\frac{1}{2}} e^{-\gamma \lambda_n^{\frac{1}{2}} \frac{\theta_*}{diam (\Omega)}|x_0-P_n^i|}.\]
Since $x_0$ was arbitrary,  the two cases together imply the existence of constants $C_1, C_2>0$ (depending on $\gamma, \theta_*$, ${\rm diam} (\Omega)$) such that
\begin{equation*}
	|u_n(x)|\leq C_1 e^{C_1 R}\lambda_n^{\frac{1}{2}}\sum_{i=1}^k e^{-C_2\lambda_n^{\frac{1}{2}}|x-P_n^i|}, ~\forall x\in \Lambda_n.
\end{equation*}
The same estimate holds for $\{v_n\}$:
\begin{equation*}
	|v_n(x)|\leq C_1 e^{C_1 R}\lambda_n^{\frac{1}{2}}\sum_{i=1}^k e^{-C_2\lambda_n^{\frac{1}{2}}|x-P_n^i|}, ~\forall x\in \Lambda_n.
\end{equation*}
Thus \eqref{globalbegrad eq} is proved, completing the proof. 
\end{proof}

\section{Proof of Theorem \ref{Thm1.1}}\label{proof of theorem}

In this section, we finalize the proof of Theorem \ref{Thm1.1}. We have constructed a sequence of mountain pass type critical points $\{(u_{\rho_n}, v_{\rho_n})\}$ for $J_{\rho_n}$ on $\mathcal{S}_a$, which possess a uniformly bounded Morse index as $\rho_n\to 1^-$.
By employing a detailed blow-up analysis, we proceed to show the boundedness of both $\{\lambda_n\}$ and $\{(u_{\rho_n},v_{\rho_n})\}$.
 \begin{prop}\label{compactnessgrad}
	Let $\{(u_n,v_n)\}\subset H^1(\Omega)\times H^1(\Omega)$ be a sequence of solutions to \eqref{withngrad} for some $\{\lambda_n\}\subset \mathbb{R}$ and $\rho_n\to 1^-$. Suppose that
	\[\int_{\Omega}\left(|u_n|^2+|v_n|^2 \right)dx=a~~ \text{and}~~m(u_n,v_n)\leq 2,~\forall n\in\mathbb{N},\]
	for some constant $a>0$, and that the energy levels  $\{c_n:=J_{\rho_n}(u_n,v_n)\}$ are bounded.
	Then the sequences $\{\lambda_n\}\subset \mathbb{R}$ and $\{(u_n,v_n)\}\subset H^1(\Omega)\times H^1(\Omega)$ are both bounded.
	Furthermore, $\{(u_n,v_n) \}$ is a bounded Palais-Smale sequence for $J$ on $\mathcal{S}_a$ at the level $c_1$.
\end{prop}
\begin{proof}
Suppose that $\lambda_n\to +\infty$.
By Theorem \ref{globalbegrad}, there exist at most $k$ points, denoted as $\{P_n^1\},\cdots,\{P_n^k\}$, where $k\leq 2$.	
In what follows, we let  $(\bar{u}_n^i, \bar{v}_n^i)$  represent the scaled sequence centered around $P_n^i$ and
$(\bar{u}^i, \bar{v}^i)$ denote the limits of $(\bar{u}_n^i, \bar{v}_n^i)$ as $n\to+\infty$.

For such blow-up points, it is possible that either
\begin{equation*}
	\lambda_n^{\frac{1}{2}}dist(P_n^i,\partial \Omega)\to +\infty~~ \mbox{or}~~\lambda_n^{\frac{1}{2}}dist(P_n^i,\partial \Omega)\to d\geq 0.
\end{equation*}
Without loss of generality, we assume there exists an integer $k_1\in [0, k]$ such that for every $i\in \{1,\cdots, k_1\}$, we have $\lambda_n^{\frac{1}{2}}dist(P_n^i,\partial \Omega)\to +\infty$, and for every $j\in \{k_1+1,\cdots, k\}$, we have $\lambda_n^{\frac{1}{2}}dist(P_n^j,\partial \Omega)\to d\geq 0$.

Note that
 \begin{equation*}
		\bigg| \lambda_n ^{\frac{1}{2}}\int_{\Omega}\left(u_n^2 +v_n^2\right) dx \bigg|=\lambda_n ^{\frac{1}{2}}a\to +\infty.
	\end{equation*}
Using a similar argument as in Theorem \ref{localbegrad}, we have
	\begin{align*}
	&\sum_{i=1}^{k_1}\int_{B_R(0)}|\bar{u}_n^i|^2 dx\to \sum_{i=1}^k \int_{B_R(0)} |\bar{u}^i|^2 dx,\\
		&\sum_{i=1}^{k_1}\int_{B_R(0)}|\bar{v}_n^i|^2 dx\to \sum_{i=1}^k \int_{B_R(0)} |\bar{v}^i|^2 dx,
	\end{align*}
and
	\begin{equation*}
	\sum_{i=k_1+1}^{k}\int_{B_R(0)\cap \Omega_n}|\bar{u}_n^i|^2 dx\to \sum_{i=1}^k \int_{B_R(0)\cap \mathbb{R}_+^N} |\bar{u}^i|^2 dx,
	\end{equation*}
	\begin{equation*}
		\sum_{i=k_1+1}^{k}\int_{B_R(0)\cap \Omega_n}|\bar{v}_n^i|^2 dx\to \sum_{i=1}^k \int_{B_R(0)\cap \mathbb{R}_+^N} |\bar{v}^i|^2 dx.
	\end{equation*}
Therefore, for any $R>0$, it follows that $n\to+\infty$,
		\begin{align}\label{esitimategrad}
		&\bigg| \lambda_n ^{\frac{1}{2}}\int_{\Omega}\left(|u_n|^2 +|v_n|^2\right) dx-\sum_{i=1}^{k_1}\int_{B_R(0)}\left(|\bar{u}_n^i|^2 +|\bar{v}_n^i|^2 \right) dx\nonumber\\
    &-\sum_{i=k_1+1}^{k}\int_{B_R(0)\cap \Omega_n}\left(|\bar{u}_n^i|^2 +|\bar{v}_n^i|^2 \right)dx \bigg|\to +\infty.
	\end{align}
In view of Theorem \ref{globalbegrad}, there exists a constant $C>0$ such that
\begin{eqnarray}
\begin{aligned}
     &\bigg| \lambda_n ^{\frac{1}{2}}\int_{\Omega}u_n^2 dx
     -\sum_{i=1}^{k_1}\int_{B_R(0)}|\bar{u}_n^i|^2 dx-\sum_{i=k_1+1}^{k}\int_{B_R(0)\cap \Omega_n}|\bar{u}_n^i|^2 dx \bigg|\\
     &=\lambda_n ^{\frac{1}{2}}
         \bigg|\int_{\Omega}u_n^2 dx-\sum_{i=1}^{k_1}\int_{B_{R\lambda_n^{-\frac{1}{2}}}(P_n^i)}u_n^2 dx-\sum_{i=k_1+1}^{k}\int_{B_{R\lambda_n^{-\frac{1}{2}}}(P_n^i)\cap \Omega}u_n^2 dx \bigg|\\
     &= \lambda_n^{\frac{1}{2}}\int_{\Omega\backslash \cup_{i=1}^{k} (B_{R\lambda_n^{-\frac{1}{2}}}(P_n^i)\cap \Omega)}u_n^2 dx\\
     &\leq C_1^2 e^{2C_1 R}\lambda_n^{\frac{3}{2}}\sum_{i=1}^{k}\int_{\Omega\backslash \cup_{i=1}^{k}(B_{R\lambda_n^{-\frac{1}{2}}}(P_n^i)\cap \Omega)}e^{-2C_2\lambda_n^{\frac{1}{2}}|x-P_n^i|} dx\\
     &\leq C e^{C R}\sum_{i=1}^{k}\int_{\mathbb{R}^N\backslash B_R(0)} e^{-2C_2|x|} dx \leq Ce^{-C'R}.
\end{aligned}
\end{eqnarray}
By using a similar argument, we have
\[\bigg| \lambda_n ^{\frac{1}{2}}\int_{\Omega}v_n^2 dx
     -\sum_{i=1}^{k_1}\int_{B_R(0)}|\bar{v}_n^i|^2 dx-\sum_{i=k_1+1}^{k}\int_{B_R(0)\cap \Omega_n}|\bar{v}_n^i|^2 dx \bigg| \leq Ce^{-C'R}.\]
Taking the limit as $n\to +\infty$, we obtain
\begin{align*}
&\limsup_{n\to +\infty} \bigg| \lambda_n ^{\frac{1}{2}}\int_{\Omega}\left(|u_n|^2+|v_n|^2 \right)dx
     -\sum_{i=1}^{k_1}\int_{B_R(0)}\left(|\bar{u}_n^i|^2+|\bar{v}_n^i|^2 \right)dx\\
     &-\sum_{i=k_1+1}^{k}\int_{B_R(0)\cap\Omega_n}\left(|\bar{u}_n^i|^2+|\bar{v}_n^i|^2 \right)dx \bigg| \leq Ce^{-C'R},
     \end{align*}
which contradicts to \eqref{esitimategrad}.	
Hence, the sequence $\{\lambda_n\}$ is bounded. Thus, by standard arguments we deduce that $\{u_n\}$ and $\{v_n \}$ are both bounded in $H^1(\Omega)$. 
\end{proof}

\begin{proof}[Completion of proof of Theorem \ref*{Thm1.1}]
By Proposition \ref{compactnessgrad} and the compact embedding of $H^1(\Omega)\hookrightarrow L^r(\Omega)$ for $r\in [1,2^*)$, we can infer that for every $a\in (0, a^*)$, the sequence $(u_{\rho_n},v_{\rho_n})$ converges strongly to $(u^*,v^*)$ in $H^1(\Omega)\times H^1(\Omega)$ for some $\lambda^*\in \mathbb{R}$. This implies that $(u^*,v^*)$ is a normalized solution to \eqref{grad} of mountain pass type.
Moreover, similar to the proof of Lemma \ref{nontrivial sol}, it follows that for any $a\in (0, \hat a^*)$, the pair $(u^*,v^*)$ is a nontrivial solution.
This completes our proof.
\end{proof}

\section{Appendix}\label{liouville thm}
In this appendix,  we prove Propositions \ref{wholeliou} and \ref{halfliou} stated in Section \ref{blow up ana}. For brevity, we prove we give the details only for Proposition \ref{halfliou}; the proof of Proposition \ref{wholeliou} follows similarly.

Let $D$ be a domain in $\mathbb{R}^3$ and assume $(u,v)\in H_{loc}^1(D)\times H_{loc}^1(D)$ satisfies
\begin{eqnarray*}
\begin{cases}
-\Delta u= uv^2,  & x\in D,\\
-\Delta v=   u^2v,  & x\in D,\\
\frac{\partial u}{\partial \nu}=\frac{\partial v}{\partial \nu}=0, &x\in \partial D.
\end{cases}
\end{eqnarray*}
Consider the quadratic form
\begin{equation*}
\begin{split}
		\tilde{Q}_{u,v}(\phi,\psi;D):= &\int_{D}\left(|\nabla \phi |^2+|\nabla \psi|^2 \right)dx 
- \int_{D} \left(v^2\phi^2+ u^2\psi^2+ 4 uv\phi\psi \right)dx.
\end{split}
\end{equation*}
We say that $(u,v)$ has finite Morse index $k$ if $k$ is the maximum dimension of a subspace $W\subset C_c^1(D)\times C_c^1(D)$ such that $\tilde{Q}_{u,v}(\phi,\psi;D)<0$ for every nonzero $(\phi,\psi) \in W\backslash$.

We introduce a cut-off function
\begin{equation*}
\varphi_{r_1, r_2}(x):=
	\begin{cases}
		0,~~ |x|<r_1~ \mbox{or} ~|x|> 2r_2,\\
		1,~~ 2r_1 \leq |x|\leq r_2,
	\end{cases}
\end{equation*}
where $0<2r_1<r_2$. Moreover, 
\begin{eqnarray*}
|\nabla  \varphi_{r_1, r_2}|\leq \frac{2}{r_1} ~~\mbox{for}~~ r_1< |x|<2r_1,~~~~|\nabla  \varphi_{r_1, r_2}|\leq \frac{2}{r_2}~~\mbox{for}~~r_2< |x|<2r_2.
\end{eqnarray*}

\begin{lem}\label{wuqing}
	Let $(u,v)$ be a solution of \eqref{tdxitong} with finite Morse index. Then there exists $R_0>0$ such that
		\begin{equation}\label{fannco}
	\tilde{Q}_{u,v}(v\varphi_{R_0, R}, u\varphi_{R_0, R}; \mathbb{R}_+^3)\geq 0
	\end{equation}
	for any $R> R_0$.
\end{lem}
\begin{proof}
Without loss of generality, we suppose that the Morse index of $(u,v)$ is $k$. By contradiction, suppose that  \eqref{fannco} does not hold. Then there exist $r_1>0$ and $R_1> 2r_1$ such that
	 $\tilde Q_{u,v}(v\varphi_{r_1, R_1}, u\varphi_{r_1, R_1}; \mathbb{R}^3_+)<0$. 
	Similarly, we can find $r_2 > 2R_1$ and $R_2 > 2r_2$, such that $\tilde{Q}_{u,v}(v\varphi_{r_2, R_2}, u\varphi_{r_2, R_2}; \mathbb{R}^3_+)<0$.
	 
By iterating this process, we can obtain sequences $r_{i+1} > 2R_{i}$ and $R_{i+1} > 2r_{i+1}$ for $i=1,2,\cdots, k$, such that
	 \[\tilde{Q}_{u,v}(v\varphi_{r_{i+1}, R_{i+1}}, u\varphi_{r_{i+1}, R_{i+1}}; \mathbb{R}^3_+)<0.\]
Due to the choice of $\{r_i, R_i\}$, the functions $\{v\varphi_{r_i, R_i}\}_{i=1}^{k+1}$ and $\{u\varphi_{r_i, R_i}\}_{i=1}^{k+1}$ admits disjoint supports. Consequently, the above inequalities imply that the Morse index of $(u,v)$ is at least $k+1$. 
Thus, the proof is complete.
\end{proof}

\begin{lem}\label{hcwj}
	Let $(u,v)$ be a bounded solution of \eqref{tdxitong} with finite Morse index. Then
	\begin{equation}\label{ruyi}
		\int_{\mathbb{R}^3_+}\left( |\nabla u|^2 + |\nabla v|^2 \right)dx<\infty, ~~ \int_{\mathbb{R}^3_+} u^2v^2 dx <\infty.
	\end{equation}
\end{lem}
\begin{proof}
	First, we define a cut-off function $\phi_R$ in $\mathbb{R}^3_+$.
Let $R_0$ be the constant defined in Lemma \ref{wuqing}. For any $R>R_0$, we define $\phi_R$ as follows:
\begin{equation*}
\phi_R(x) :=
\begin{cases}
0, & \text{if } |x| \leq R_0 \text{ or } |x| \geq 2R, \quad x \in \mathbb{R}^3_+, \\
1, & \text{if } 2R_0 \leq |x| \leq R, \quad x \in \mathbb{R}^3_+,
\end{cases}
\end{equation*}
with the additional property that$|\nabla \phi_R|\leq \frac{2}{R_0}$ for $R_0\leq |x|\leq 2R_0$ and $|\nabla \phi_R|\leq \frac{2}{R}$ for $R\leq |x|\leq 2R$.
 
By Lemma \ref{wuqing}, we have $\tilde{Q}_{u,v}(v\phi_R, u\phi_R; \mathbb{R}_+^3)\geq 0$.
This implies
 \begin{align}\label{12-31-1}
 &\int_{\mathbb{R}^3_+}\left[ v^4\phi_R^2+ u^4\phi_R^2 +4 u^2v^2\phi_R^2 \right]dx\nonumber\\
 &\leq \int_{\mathbb{R}^3_+}\left( |\nabla (v\phi_R)|^2 + |\nabla (u\phi_R)|^2 \right)dx\nonumber\\
&=\int_{\mathbb{R}^3_+}\left[ \phi_R^2 |\nabla v|^2   +2v\phi_R\nabla v\nabla \phi_R+ v^2|\nabla \phi_R|^2 \right]dx\nonumber\\
&\quad+\int_{\mathbb{R}^3_+} \left[\phi_R^2 |\nabla u|^2   +2u\phi_R\nabla u\nabla \phi_R+ u^2|\nabla \phi_R|^2\right]dx.    
 \end{align}
Taking $(u\phi_R^2, v\phi_R^2)$ as a test function in \eqref{tdxitong}, we get
 \begin{align*}\label{tianya}
 	 	 &\int_{\mathbb{R}^3_+} 2u^2v^2 \phi_R^2 dx\\
     &=\int_{\mathbb{R}^3_+}\left[ \phi_R^2 |\nabla u|^2 +2\phi_R u\nabla u\nabla \phi_R +\phi_R^2 |\nabla v|^2 +2\phi_R v\nabla v\nabla \phi_R \right]dx.
 \end{align*}
Substituting it into \eqref{12-31-1}, it follows that
\begin{align*}
	 &\int_{\mathbb{R}^3_+} (u^2+ v^2)^2\phi_R^2 dx\\
	 &\leq \int_{\mathbb{R}^3_+} \left(u^2|\nabla \phi_R|^2 + v^2 |\nabla \phi_R|^2 \right) dx\\
	 &\leq \int_{\{R_0\leq |x|\leq 2R_0\}\cap \mathbb{R}^3_+} (u^2+v^2) |\nabla \phi_R|^2 dx+\int_{\{R\leq |x|\leq 2R\}\cap \mathbb{R}^3_+} (u^2+v^2) |\nabla \phi_R|^2 dx\\
	 &\leq c_0+\left(\int_{\{R\leq |x|\leq 2R\}\cap \mathbb{R}^3_+} (u^2+v^2)^2 dx\right)^{\frac{1}{2}} \left(\int_{\{R\leq |x|\leq 2R\}\cap \mathbb{R}^3_+} |\nabla \phi_R|^4 dx\right)^{\frac{1}{2}}\\
	 &\leq c_0 + C\left(\int_{\{R\leq |x|\leq 2R\}\cap \mathbb{R}^3_+} (u^2+v^2)^2 dx\right)^{\frac{1}{2}}R^{-\frac{1}{2}}.
\end{align*}
If $\int_{\mathbb{R}^3_+} (u^2 + v^2)^2 dx=\infty$, then
 \[\int_{B_R^+} (u^2 + v^2)^2 dx\leq C\left(\int_{B_{2R}^+} (u^2+v^2)^2 dx\right)^{\frac{1}{2}}R^{-\frac{1}{2}},\]
where $B_R^+=\{x\in \mathbb{R}_+^3: |x|<R\}$.
Let $H(R)= \int_{B_R^+} (u^2 + v^2)^2 dx$.  Iterating the above inequality, we obtain
\begin{equation}\label{fengyu}
	H(R)\leq C R^{-\frac{1}{2}\alpha} H(2^{k+1}R)^{(\frac{1}{2})^{k+1}},
\end{equation} 
where $\alpha= 1+\frac{1}{2}+\frac{1}{2^2}+\cdots+\frac{1}{2^k}$.
Since $u$ and $v$ are bounded, we deduce that \[ H(R)\leq C R^{-\frac{1}{2}\alpha + \frac{3}{2^{k+1}}} .\]
By choosing $k$ large enough, we can ensure that  $-\frac{1}{2}\alpha + \frac{3}{2^{k+1}}<0$.
Thus, it follows from \eqref{fengyu} that $H(R)\to 0$ as $R\to \infty$,  which leads to a contradiction.
Therefore, we conclude that $\int_{\mathbb{R}^3_+} (u^2 + v^2)^2 dx<\infty$, which gives that $\int_{\mathbb{R}^3_+} u^2v^2 dx <\infty$.
Furthermore, from the system \eqref{tdxitong}, we also have
\begin{equation*}
	\int_{\mathbb{R}^3_+ }\left(|\nabla u|^2 +|\nabla v|^2 \right)dx< \infty. 
\end{equation*}
Hence (\ref{ruyi}) holds.
\end{proof}

\begin{lem}\label{kexn}
Let $(u,v)$ be a solution of \eqref{tdxitong}. Then the following identity holds
\begin{align}\label{locaPoho}
		&\frac{1}{2}\int_{B_R^+}\left(|\nabla u|^2+|\nabla v|^2 \right)dx- \frac{3}{2} \int_{B_R^+} u^2v^2 dx\nonumber\\
		 &=\frac{R}{2}\int_{\partial B_R^+}\left(|\nabla u|^2+|\nabla v|^2 \right)d\sigma 
		- R\int_{\partial B_R^+}\left(\left|\frac{\partial u}{\partial \nu} \right|^2+ \left|\frac{\partial v}{\partial \nu} \right|^2 \right)d\sigma\nonumber\\
    &\quad - \frac{R}{2} \int_{\partial B_R^+} u^2v^2 d\sigma.
\end{align}
\end{lem}
\begin{proof}
By taking $(x\cdot \nabla u, x\cdot \nabla v)$ as the test function, it follows from \eqref{tdxitong} that
	\begin{align*}
		&-\int_{B_R^+} \Delta u (x\cdot \nabla u) dx \\
    &=\int_{B_R^+} \nabla u  \nabla(x\cdot \nabla u) dx-\int_{\partial B_R^+}(x\cdot \nabla u)(\nu\cdot \nabla u) d\sigma \\
		&= \int_{B_R^+} \nabla u  \nabla(x\cdot \nabla u) dx- R\int_{\partial B_R^+} \left|\frac{\partial u}{\partial \nu} \right|^2 d\sigma\\
		&=\int_{B_R^+} |\nabla u|^2 dx-\frac{3}{2}\int_{B_R^+} |\nabla u|^2 dx+ \frac{R}{2}\int_{\partial B_R^+} |\nabla u|^2 d\sigma-R\int_{\partial B_R^+} \left|\frac{\partial u}{\partial \nu} \right|^2 d\sigma\\
		&=-\frac{1}{2}\int_{B_R^+} |\nabla u|^2 dx+\frac{R}{2}\int_{\partial B_R^+} |\nabla u|^2 d\sigma-R\int_{\partial B_R^+} \left|\frac{\partial u}{\partial \nu} \right|^2 d\sigma,
	\end{align*}
and
\begin{align*}
		&-\int_{B_R^+} \Delta v (x\cdot \nabla v) dx\\
    &=-\frac{1}{2}\int_{B_R^+} |\nabla v|^2 dx+\frac{R}{2}\int_{\partial B_R^+} |\nabla v|^2 d\sigma-R\int_{\partial B_R^+} \left|\frac{\partial v}{\partial \nu} \right|^2 d\sigma.
	\end{align*}
Next, considering the term involving $uv^2$ and $u^2v$, we deduce that
\begin{align*}
		\int_{B_R^+}\left[ uv^2 (x\cdot \nabla u) +u^2v (x\cdot \nabla v)\right] dx=-\frac{3}{2}\int_{B_R^+} u^2v^2 dx+\frac{R}{2}\int_{\partial B_R^+} u^2v^2 d\sigma.
	\end{align*}	
Combining these results, we obtain \eqref{locaPoho} and this completes the proof.
\end{proof}

\begin{proof}[Proof of Proposition \ref{halfliou}.]
By Lemma \ref{hcwj}, we have
\begin{equation*}
	\int_{\mathbb{R}^3_+}\left( |\nabla u|^2 + |\nabla v|^2 \right) dx<\infty, ~~ \int_{\mathbb{R}^3_+} (u^2 + v^2)^2 dx <\infty.
\end{equation*}
Then there exists a sequence $\{R_n\}\subset\mathbb{R}^+$ with $R_n\to \infty$ such that
\begin{align*}
  &\frac{R_n}{2}\int_{\partial B_{R_n}^+}\left(|\nabla u|^2+|\nabla v|^2\right) d\sigma 
		- R_n\int_{\partial B_{R_n}^+}\left(\left|\frac{\partial u}{\partial \nu} \right|^2+ \left|\frac{\partial v}{\partial \nu} \right|^2\right) d\sigma \\
    &- \frac{R_n}{2} \int_{\partial B_{R_n}^+} u^2v^2 d\sigma\to 0. 
\end{align*}
Hence, using \eqref{locaPoho} we obtain
\[\frac{1}{2}\int_{\mathbb{R}^3_+}\left(|\nabla u|^2+|\nabla v|^2\right)dx - \frac{3}{2} \int_{\mathbb{R}^3_+} u^2v^2 dx=0. \]
On the other hand, by virtue of \eqref{tdxitong}, we deduce
\[\int_{\mathbb{R}^3_+}\left( |\nabla u|^2 + |\nabla v|^2 \right)dx =2\int_{\mathbb{R}^3_+} u^2v^2  dx.\]
Therefore, it follows that
\begin{equation*}
	\int_{\mathbb{R}^3_+}\left( |\nabla u|^2 + |\nabla v|^2 \right)dx=\int_{\mathbb{R}^3_+} u^2v^2  dx=0,
\end{equation*}
which implies that $u\equiv v\equiv 0$. The proof is complete.
\end{proof}

\section*{Acknowledgements}
The research of Xiaojun Chang is partially supported by the Research Project of the Education Department of Jilin Province (JJKH20250296KJ), NSFC (12471102) and NSF of Jilin Province (20250102004JC). Yohei Sato is partially supported by JSPS KAKENHI(JP20K03691).


\subsection*{Author contributions}
All authors contributed equally to this work.

\subsection*{Data availability}
Data sharing not applicable to this article as no data sets were generated or analysed during
the current study.

\subsection*{Declarations}
The authors declare that they have no conflict of interest.


\begin{thebibliography}{99}
 \bibitem{Alves-2007} Alves, C. O.:
 	Local mountain pass for a class of elliptic system.
 	J. Math. Anal. Appl. \textbf{335}, 135-150 (2007)

  \bibitem{AJM2022}
Alves, C. O., Ji, C., Miyagaki, O. H.:
Normalized solutions for a Schr\"{o}dinger equation with critical growth in $\mathbb{R}^{N}$.
Calc. Var. Partial Differential Equations \textbf{61}(1), 24 (2022)

 \bibitem{Alves-2005}  
 Alves, C. O., Soares, S. H. M.:
Existence and concentration of positive solutions for a class of gradient systems.
NoDEA Nonlinear Differential Equations Appl. \textbf{12}, 437-457 (2005)
  

\bibitem{BL-1992} 
Bahri, A., Lions, P. L.:
Solutions of superlinear elliptic equations and their Morse indices.
Comm. Pure Appl. Math. \textbf{45}, 1205-1215 (1992)

\bibitem{BJ2018}  
Bartsch, T., Jeanjean, L.:
Normalized solutions for nonlinear Schr\"odinger systems.
Proc. Roy. Soc. Edinburgh Sect. A. \textbf{148}, 225-242 (2018)



\bibitem{BJS2016}  
Bartsch, T., Jeanjean, L., Soave, N.:
Normalized solutions for a system of coupled cubic Schr\"odinger equations on $\mathbb{R}^3$.
J. Math. Pures Appl. (9) \textbf{106}, 583-614 (2016)

\bibitem{BMRV2021}  
Bartsch, T., Molle, R., Rizzi, M., Verzini, G.:
Normalized solutions of mass supercritical Schr\"odinger equations with potential.
Comm. Partial Differential Equations \textbf{46}, 1729-1756 (2021)


\bibitem{Bartsch2017}  
Bartsch, T., Soave, N.:
A natural constraint approach to normalized solutions of nonlinear {S}chr\"odinger equations and systems.
J. Funct. Anal. \textbf{272}, 4998-5037 (2017)

\bibitem{BS2019} 
Bartsch, T., Soave, N.:
Multiple normalized solutions for a competing system of Schr\"odinger equations.
Calc. Var. Partial Differential Equations \textbf{58}, Paper No. 22, 24 pp (2019)


\bibitem{BQZ2024}  
Bartsch, T., Qi, S. J., Zou, W. M.:
Normalized solutions to Schr\"odinger equations with potential and inhomogeneous nonlinearities on large smooth domains.
Math. Ann. \textbf{390}, 4813-4859 (2024)


\bibitem{BZW-2021}  
Bartsch, T., Zhong, X. X., Zou, W. M.:
Normalized solutions for a coupled Schrödinger system.
Math. Ann. \textbf{380}, 1713-1740 (2021)

\bibitem{Berestycki-1983} 
Berestycki, H., Lions, P. L.:
Nonlinear scalar field equations, II: Existence of infinitely many solutions.
Arch. Ration. Mech. Anal. \textbf{82}, 347-375 (1983)


\bibitem{BM2020} 
Bieganowski, B., Mederski, J.:
Normalized ground states of the nonlinear Schr\"{o}dinger equation with at least mass critical growth.
J. Funct. Anal. \textbf{280}, 108989 (2021)


\bibitem{BCJS2024} 
Borthwick, J., Chang, X. J., Jeanjean, L., Soave, N.:
Bounded Palais-Smale sequences with Morse type information for some constrained functionals.
Trans. Amer. Math. Soc. \textbf{377}, 4481-4517 (2024)



\bibitem{BCJS2023} 
Borthwick, J., Chang, X. J., Jeanjean, L., Soave, N.:
Normalized solutions of $L^2$-supercritical NLS equations on noncompact metric graphs with localized nonlinearities.
Nonlinearity \textbf{36}, 3776-3795 (2023) 

\bibitem{CL1}  
Cazenave, T., Lions, P. L.:
Orbital stability of standing waves for some nonlinear {S}chr\"{o}dinger equations.
Comm. Math. Phys. \textbf{85}, 549-561 (1982)

\bibitem{CLY2025}  
Chang, X. J., Liu, M. T., Yan, D. K.:
Positive normalized solutions of Schr\"odinger equations with Sobolev critical growth in bounded domains.
Preprint \href{https://arxiv.org/abs/2505.07578}{arXiv:2505.07578} (2025)

\bibitem{CJS2024} 
Chang, X. J., Jeanjean, L., Soave, N.:
Normalized solutions of $L^2$-supercritical {NLS} equations on compact metric graphs.
Ann. Inst. H. Poincar\'e (C) Anal. Non Lin\'eaire \textbf{41}, 933-959 (2024)



\bibitem{CRZ2025} 
Chang, X. J., R\u{a}dulescu, V. D., Zhang, Y. X.:
Solutions with prescribed mass for $L^2$-supercritical NLS equations under Neumann boundary conditions.
Proc. Roy. Soc. Edinburgh Sect. A. (2025). \href{https://doi.org/10.1017/prm.2025.24}{https://doi.org/10.1017/prm.2025.24}


\bibitem{CS-2019} 
Chang, X. J., Sato, Y.:
Localized solutions of nonlinear Schr\"odinger systems with critical frequency for infinite attractive case.
NoDEA Nonlinear Differential Equations Appl. \textbf{26}, 1-31 (2019)


\bibitem{CS-2020} 
Chang, X. J., Sato, Y.:
Multiplicity of localized solutions of nonlinear Schr\"odinger systems for infinite attractive case.
J. Math. Anal. Appl. \textbf{491}, 124358, 17 pp (2020)

 \bibitem{CT2024}
Chen, S. T., Tang, X. H.:
Another look at Schr\"odinger equations with prescribed mass.
J. Differential Equations \textbf{386}, 435-479 (2024)

\bibitem{CDGS2025} 
De Coster, C., Dovetta, S., Galant, D., Serra, E.:
An action approach to nodal and least energy normalized solutions for nonlinear Schr\"odinger equations.
Ann. Inst. H. Poincar\'e (C) Anal. Non Lin\'eaire (2025). \href{https://doi.org/10.4171/AIHPC/160}{https://doi.org/10.4171/AIHPC/160}


\bibitem{DJS2024} 
Dovetta, S., Jeanjean, L., Serra, E.:
Normalized solutions of $L^2$-supercritical NLS equations on noncompact metric graphs.
Proc. Roy. Soc. Edinburgh Sect. A. (2025). \href{https://doi.org/10.1017/prm.2025.29}{https://doi.org/10.1017/prm.2025.29}


\bibitem{DST2025}
Dovetta, S., Serra, E., Tentarelli, L.: 
Non-uniqueness of normalized NLS ground states on polygons with homogeneous Neumann boundary conditions.
Discrete Continuous Dyn. Syst. (2025). \href{https://doi.org/10.3934/dcds.2025182}{https://doi.org/10.3934/dcds.2025182}

\bibitem{EP2011} 
Esposito, P., Petralla, M.:
Pointwise blow-up phenomena for a Dirichlet problem.
Comm. Partial Differential Equations \textbf{36}, 1654-1682 (2011)

\bibitem{F-2007} 
Farina, A.:
On the classification of solutions of the Lane-Emden equation on unbounded domains of $\mathbb{R}^N$.
J. Math. Pures Appl. (9) \textbf{87}, 537-561 (2007)

\bibitem{F2010} 
Frantzeskakis, D. J.:
Dark solitons in atomic Bose-Einstein condensates: from theory to experiments.
J. Phys. A: Math. Theor. \textbf{43}, Paper No. 68, 82 pp (2010)

\bibitem{GJ2018} 
Gou, T. X., Jeanjean, L.:
Multiple positive normalized solutions for nonlinear Schr\"odinger systems.
Nonlinearity \textbf{31}, 2319-2345 (2018)

\bibitem{GH2025} 
Guo, Q. D., Hua, Q. Q.:
Segregated solutions for a nonlinear Schr\"odinger system involving mass supercritical exponents. 
Nonlinearity \textbf{38}, 035009, 45pp (2025)

\bibitem{GY2023}  
Guo, Q., Yang, J.:
Excited states for two-component Bose-Einstein condensates in dimension two.
J. Differential Equations \textbf{343}, 659-686 (2023)



\bibitem{GuoLW-2019}  
Guo, Y. J., Li, S., Wei, J. C., Zeng, X. Y.:
Ground states of two-component attractive Bose-Einstein condensates I: Existence and uniqueness.
J. Funct. Anal. \textbf{276}, 183-230 (2019)

\bibitem{GuoLW2019TAMS}  
Guo, Y. J., Li, S., Wei, J. C., Zeng, X. Y.:
Ground states of two-component attractive Bose-Einstein condensates II: Semi-trivial limit behavior.
Trans. Amer. Math. Soc. \textbf{371}, 6903-6948 (2019)


\bibitem{HRS-1998-1}  
Harrabi, A., Rebhi, S., Selmi, S.:
Solutions of superlinear equations and their Morse indices I.
Duke Math. J. \textbf{94}, 141-157 (1998)


\bibitem{HRS-1998-2} 
Harrabi, A., Rebhi, S., Selmi, S.:
Solutions of superlinear equations and their Morse indices II.
Duke Math. J. \textbf{94}, 159-179 (1998)

 


\bibitem{IM2020} 
Ikoma, N., Miyamoto, Y.:
Stable standing waves of nonlinear Schr\"odinger equations with potentials and general nonlinearities.
Calc. Var. Partial Differential Equations \textbf{59}, Paper No. 48, 20 pp (2020)


\bibitem{IkNo} 
Ikoma, N., Tanaka, K.:
A note on deformation argument for $L^2$ normalized solutions of nonlinear {S}chr\"odinger equations and systems.
Adv. Differential Equations \textbf{24}, 609-646 (2019)


\bibitem{Jeanjean1997} 
Jeanjean, L.:
Existence of solutions with prescribed norm for semilinear elliptic equations.
Nonlinear Anal. \textbf{28}, 1633-1659 (1997)





\bibitem{JJLV2022}  
Jeanjean,L., Jendrej, J., Le, T. T., Visciglia, N.:
Orbital stability of ground states for a Sobolev critical Schr\"odinger equation.
J. Math. Pures Appl. (9) \textbf{164}, 158-179 (2022)


\bibitem{JL2022}  
Jeanjean, L., Le, T. T.:
Multiple normalized solutions for a Sobolev critical Schr\"odinger equations.
Math. Ann. \textbf{384}, 101-134 (2022)


\bibitem{JL2020} 
Jeanjean, L., Lu, S. S.:
A mass supercritical problem revisited.
Calc. Var. Partial Differential Equations \textbf{59}, Paper No. 174, 43 pp (2020)


\bibitem{JZZ2024}  
Jeanjean, L., Zhang, J. J., Zhong, X. X.:
Normalized ground states for a coupled Schr\"odinger system: Mass super-critical case.
NoDEA Nonlinear Differential Equations Appl. \textbf{31}, Paper No. 85, 26 pp (2024)


\bibitem{LM2024}  
Lancelotti, S., Molle, R.:
Normalized positive solutions for Schr\"odinger equations with potentials in unbounded domains.
Proc. Roy. Soc. Edinburgh Sect. A. \textbf{154}, 1518-1551 (2024)

\bibitem{L-2021}  
Le, P.:
Stable and finite Morse index solutions of a nonlinear Schr\"odinger system.
NoDEA Nonlinear Differential Equations Appl. \textbf{28}, Paper No. 39, 16 pp (2021)




\bibitem{Malomed2008}  
Malomed, B.:
Multi-component Bose-Einstein condensates: theory.
P.G. Kevrekidis, D.J. Frantzeskakis, R. Carretero-Gonzalez (Eds.), Emergent Nonlinear Phenomena in Bose-Einstein Condensation, Springer-Verlag, Berlin, 287-305 (2008)

\bibitem{MRV2022}  
Molle, R., Riey, G., Verzini, G.:
Normalized solutions to mass supercritical Schr\"odinger equations with negative potential.
J. Differential Equations \textbf{333}, 302-331 (2022)




\bibitem{NT-1991} 
Ni, W. M., Takagi, I.:
On the shape of least-energy solutions to a semilinear Neumann problem.
Comm. Pure Appl. Math. \textbf{44}, 819-851 (1991)


\bibitem{NT-1993} 
Ni, W. M., Takagi, I.:
Locating the peaks of least-energy solutions to a semilinear Neumann problem.
Duke Math. J. \textbf{70}, 247-281 (1993)


\bibitem{N-1966} 
Nirenberg, L.:
An extended interpolation inequality.
Ann. Scuola Norm. Sup. Pisa Cl. Sci. (3) \textbf{20}, 733-737 (1966)

\bibitem{NTV-2014}  
Noris, B., Tavares, H., Verzini, G.:
Existence and orbital stability of the ground states with prescribed mass for the $L^2$-critical and supercritical NLS on bounded domains.
Anal. PDE \textbf{7}, 1807-1838 (2014)



\bibitem{NTV-2015} 
Noris, B., Tavares, H., Verzini, G.:
Stable solitary waves with prescribed $L^2$-mass for the cubic Schr\"odinger system with trapping potentials.
Discrete Continuous Dyn. Syst. \textbf{35}, 6085-6112 (2015)


\bibitem{NTV-2019} 
Noris, B., Tavares, H., Verzini, G.:
Normalized solutions for nonlinear Schr\"odinger systems on bounded domains.
Nonlinearity \textbf{32}, 1044-1072 (2019)


\bibitem{PPVV-2021} 
Pellacci, B., Pistoia, A., Vaira, G., Verzini, G.:
Normalized concentrating solutions to nonlinear elliptic problems.
J. Differential Equations \textbf{275}, 882-919 (2021)



\bibitem{PV-2017} 
Pierotti, D., Verzini, G.:
Normalized bound states for the nonlinear {S}chr\"odinger equation in bounded domains.
Calc. Var. Partial Differential Equations \textbf{56}, Paper No. 133, 27 pp (2017)


\bibitem{PVY-2025} 
Pierotti, D., Verzini, G., Yu, J. W.:
Normalized solutions for Sobolev critical Schr\"odinger equations on bounded domains.
SIAM J. Math. Anal. \textbf{57}, 262-285 (2025)



\bibitem{PS-2004}  
Pucci, P., Serrin, J.:
The strong maximum principle revisited.
J. Differential Equations \textbf{196}, 1-66 (2004)



\bibitem{SW-2013}
Sato, Y., Wang, Z.-Q.:
On the multiple existence of semi-positive solutions for a nonlinear Schr\"odinger system. 
Ann. Inst. H. Poincar\'e C Anal. Non Lin\'eaire \textbf{30}, 1-22 (2013)

\bibitem{SW-2015}
Sato, Y., Wang, Z.-Q.:
Least energy solutions for nonlinear Schr\"odinger systems with mixed attractive and repulsive couplings.
Adv. Nonlinear Stud. \textbf{15}, 1-22 (2015)



\bibitem{Soave-2020JDE} 
Soave, N.:
Normalized ground states for the NLS equation with combined nonlinearities.
J. Differential Equations \textbf{269}, 6941-6987 (2020)



\bibitem{Soave-2020JFA} 
Soave, N.:
Normalized ground states for the NLS equation with combined nonlinearities: the Sobolev critical case.
J. Funct. Anal. \textbf{279}, Paper No. 108610, 43 pp (2020)

\bibitem{SZ2025} 
Song, L. J., Zou, W. M.:
On multiple sign-changing normalized solutions to the Br\'ezis-Nirenberg problem.
Math. Z. \textbf{309}, Paper No. 61, 23 pp (2025)

\bibitem{WC2024}  
Wang, Q., Chang, X. J.:
Normalized solutions of $L^2$-supercritical Kirchhoff equations in bounded domains.
J. Geom. Anal. \textbf{34}, Paper No. 358, 30 pp (2024)




\bibitem{WW-2022}  
Wei, J. C., Wu, Y. Z.,
Normalized solutions for Schr\"{o}dinger equations with critical sobolev exponent and mixed nonlinearities.
J. Funct. Anal. \textbf{283}, Paper No. 109574, 46 pp (2022)

\bibitem{YZ-2019}  
Yang, H., Zou, W. M.:
Stable and finite Morse index solutions of a nonlinear elliptic system.
J. Math. Anal. Appl. \textbf{471}, 147-169 (2019)

\bibitem{ZZ2022}
Zhang, C. X., Zhang, X.: 
Normalized multi-bump solutions of nonlinear Schr\"odinger equations via variational approach.
Calc. Var. Partial Differential Equations \textbf{61}, Paper No. 57, 20 pp (2022)

 







\end{thebibliography}
\end{document}